\documentclass[11pt]{article}

\usepackage{amssymb,amsmath,amsfonts}
\usepackage{graphicx,color,enumitem}
\usepackage{amsthm} 
\usepackage{bm}
\usepackage[round]{natbib}
\usepackage{geometry}

\usepackage[colorinlistoftodos, textwidth=4cm, shadow]{todonotes}
\RequirePackage[colorlinks,citecolor=blue,urlcolor=blue]{hyperref}

\usepackage{bbm}
\usepackage{capt-of}

\newcommand{\e}{{\rm e}}

\newcommand{\E}{{\mathbb E}}

\newcommand{\Q}{{\mathbb Q}}
\newcommand{\C}{{\mathbb C}}
\newcommand{\R}{{\mathbb R}}

\newcommand{\Mid}{{\ \Big|\ }}

\newcommand{\Fc}{{\mathcal F}}

\newcommand{\id}{{\rm id}}

\DeclareMathOperator{\tr}{tr}
\DeclareMathOperator{\Tr}{Tr}

\newtheorem{theorem}{Theorem}

\newtheorem{corollary}[theorem]{Corollary}
\newtheorem{definition}[theorem]{Definition}

\newtheorem{lemma}[theorem]{Lemma}

\newtheorem{remark}[theorem]{Remark}

\theoremstyle{definition}
\newtheorem{example}[theorem]{Example}

\numberwithin{equation}{section}
\numberwithin{theorem}{section}

\definecolor{darkgreen}{rgb}{0,0.7,0}

\newcommand{\iii}{{\vert\kern-0.25ex\vert\kern-0.25ex\vert}}

\usepackage{mathtools}
\mathtoolsset{showonlyrefs}

\newcommand{\EE}{\mathbb{E}}

\begin{document}
	
	\title{The characteristic function of Gaussian stochastic volatility models: an analytic expression}
	
	\author{Eduardo Abi Jaber\thanks{Universit\'e Paris 1 Panth\'eon-Sorbonne, Centre d'Economie de la Sorbonne, 106, Boulevard de l'H\^opital, 75013 Paris, eduardo.abi-jaber@univ-paris1.fr. I would like to thank Shaun Li for interesting discussions.}}

	\maketitle
	
	\begin{abstract}
	Stochastic volatility models based on Gaussian  processes, like fractional Brownian motion, are able to reproduce important stylized facts of financial markets such as  rich autocorrelation structures, persistence and roughness of sample paths. This is made  possible by virtue of the flexibility introduced in the choice of the covariance function of the Gaussian process. The price to pay is that, in general, such models are no longer Markovian nor semimartingales, which limits their practical use. We derive, in two different ways, an explicit analytic expression for the joint characteristic function of the log-price and its integrated variance in general Gaussian stochastic volatility models. Such analytic expression can be approximated by closed form matrix expressions. This opens the door to fast approximation of the joint density and  pricing of derivatives on both the stock and its realized variance using Fourier inversion techniques. In the context of rough volatility modeling, our results apply to the (rough) fractional Stein--Stein model and provide the first analytic formulae for option pricing known to date, generalizing that of Stein--Stein, Sch\"obel--Zhu and a special case of Heston.  
 	\\[2ex] 
	\end{abstract}


\section{Introduction}

In the realm of risk management in mathematical  finance, academics and practitioners have been always striving for explicit solutions to option prices and hedging strategies in their models.   Undoubtedly,  finding  explicit expressions to a theoretical problem can be highly satisfying in itself; it also has many practical advantages such as: reducing computational time (compared to brute force Monte-Carlo simulations for instance); achieving a higher precision for option prices and hedging strategies;  providing a better understanding of the role of the parameters of the model and the sensitivities of the prices and strategies with respect to them. As one would expect, explicit expressions usually  come at the expense of sacrificing  the flexibility and the accuracy of the model. In a nutshell, the aim of the present paper is to show that  analytic expressions for option prices can be found in a highly flexible class of non-Markovian stochastic volatility models.

\subsection*{From Black-Scholes to rough volatility}
In their seminal paper, \citet{black1973pricing} derived closed form solutions  for the prices of European call and put options in the geometric Brownian motion model where the dynamics of the stock price $S$ are given by: 
\begin{align}\label{eq:introBS}
dS_t = S_t \sigma dB_t, \quad S_0>0,
\end{align}
with $B$  a standard Brownian motion and $\sigma$ the constant instantaneous volatility parameter.   Although revolutionary,  the model remains very simple: it drifts away from the reality of financial markets characterized by non-Gaussian returns, fat tails of stock prices and their volatilities, asymmetric option prices (i.e.~the implied volatility smile and skew)\ldots see \citet{cont2001empirical}.
 Since then a large and growing literature has been developed to refine the \citet{black1973pricing} model. One notable direction is stochastic volatility modeling where   the constant volatility $\sigma$  in \eqref{eq:introBS} is replaced by a Markovian stochastic process $(\sigma_t)_{t\geq 0}$. In their celebrated paper,  \citet{stein1991stock} modeled $(\sigma_t)_{t\geq 0}$ by a mean-reverting Brownian motion of the form 
\begin{align}\label{intro:ssvol}
d\sigma_t = \kappa(\theta - \sigma_t)dt + \nu dW_t, 
\end{align}
where $W$ is a standard Brownian motion independent of $B$. Remarkably, they obtained closed-form expressions for the characteristic function of the log-price, which allowed them to recover the density as well as option prices by Fourier inversion of the characteristic function. Later on the model has been extended by \citet{schobel1999stochastic} to account for the leverage effect, i.e.~an arbitrary correlation between $W$ and $B$. Similar formulas for the characteristic function of the log--price to those of Stein--Stein   are derived for the non-zero correlation case. 

Prior to the extension by \citet{schobel1999stochastic},  \citet{heston1993closed} took a slightly different approach to include the leverage effect by introducing a model  deeply rooted in the Stein--Stein model. Heston  observed that the instantenous variance process $V_t=\sigma^2_t$ in the Stein--Stein model with $\theta=0$ follows a CIR process 
thanks to It\^o's formula,\footnote{Squares of Brownian motion constitute the building blocks of squared Bessel processes, see \citet[Chapter XI]{RY:99}.}  so that the Stein--Stein model  can be recast in the following form 
\begin{align}
dS_t &= S_t \sqrt{V_t}dB_t,\\
dV_t &= (\nu^2-2\kappa V_t)dt + 2 \nu \sqrt{V_t} dW_t, \label{intro:heston}
\end{align}
 where $B= \rho W + \sqrt{1 - \rho^2} W^\perp$ with $\rho \in [-1,1]$ and $W^{\perp}$ a Brownian motion independent of $W$. Such model remains tractable as it was shown earlier in the context of bond pricing with uncertain  inflation by \citet[Equations (51)-(52)]{cox1985intertemporal}.\footnote{The long-term level of the variance $\nu^2$ in \eqref{intro:heston} can be replaced by a more general coefficient $\theta\geq 0$.} \citet{heston1993closed}  carried on by deriving closed form expressions for the characteristic function of the log--price, which made his model one of the most, if not the most, popular model  among practitioners. As one would expect, the expressions of \citet{heston1993closed} and \citet{schobel1999stochastic} share a lot of similarities and they perfectly agree when $\theta=0$ in \eqref{intro:ssvol},  see \citet[equation (44)]{lord2006rotation}. Such analytical tractability motivated the development of the theory of finite-dimensional Markovian affine processes, see \cite{DFS:03}.

Unfortunately, Markovian stochastic volatility models, such as the Heston and the Stein--Stein models, are not flexible enough: they generate an
auto-correlation structure which is too simplistic compared to empirical observations. Indeed,  several empirical studies have documented the persistence in the volatility time series, see \citet{andersen1997intraday,ding1993long}. More recently, \citet{gatheral2018volatility} and \citet{bennedsen2016decoupling} show that  the sample paths of the realized volatility are rougher than standard Brownian motion at any realistic time scale as illustrated on Figure~\ref{fig:skew intro}-(a). From a pricing perspective, continuous semi-martingale models driven by a standard Brownian motion fail to reproduce the power-law decay of  the at-the-money skew of option prices as shown on Figure~\ref{fig:skew intro}-(b), see \citet{carr2003finite, fouque2003multiscale, lee2005implied, alos2007short, bayer2016pricing, fukasawa2011asymptotic, fukasawa2021volatility}.

\begin{center}
	\includegraphics[scale=0.6]{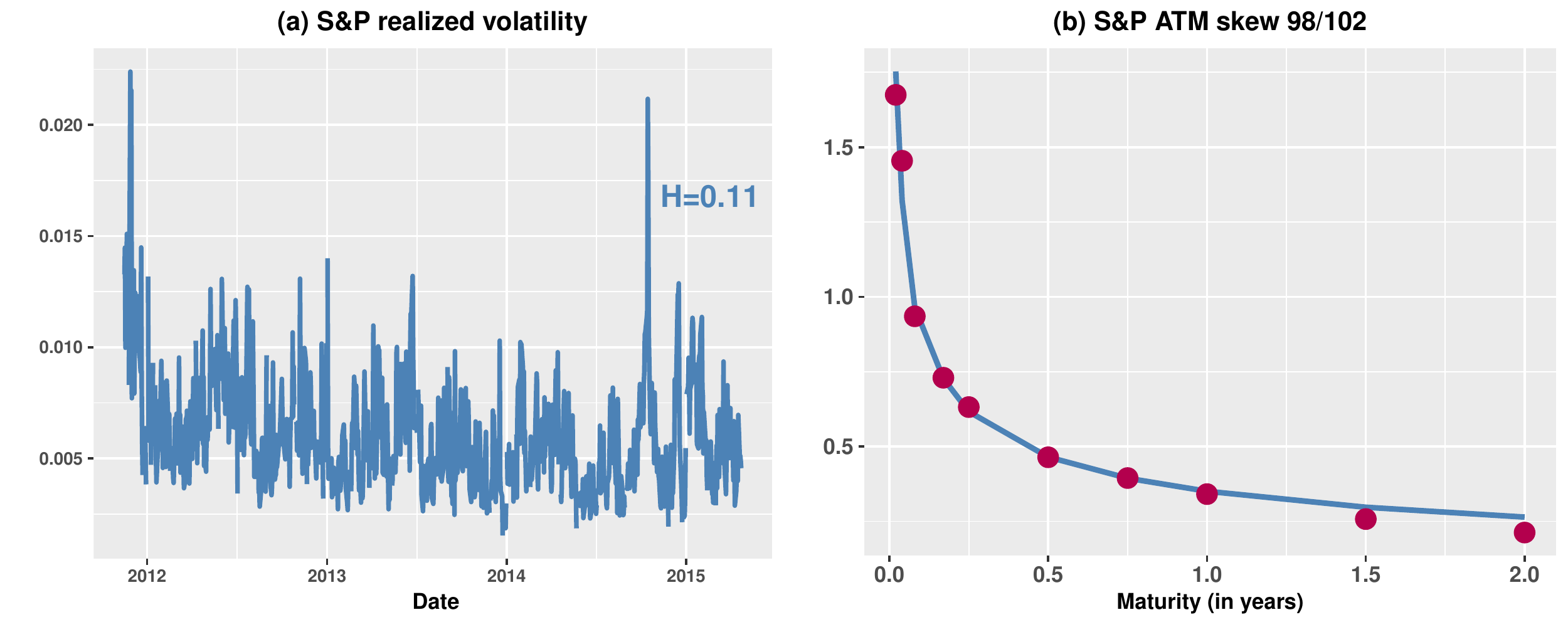}
	\captionof{figure}{(a) Realized volatility of the S\&P  downloaded from \url{https://realized.oxford-man.ox.ac.uk/} with an estimated Hurst index of $\hat H=0.11$. (b) Term structure of the at-the-money skew of the implied volatility $\frac{\partial \sigma_{\text{implicit}}(k,T)}{\partial k}\big |_{k=0} $ for the S\&P index on June 20, 2018 (red dots) and a power-law fit $t \to 0.35\times t^{-0.41}$. Here $k:=\ln(K/S_0)$ stands for the log-moneyness and $T$ for the time to maturity.}
	\label{fig:skew intro}
\end{center}

These studies have motivated the need  to enhance conventional stochastic volatility models  with richer auto-correlation structures. This has been  initiated in \citet{comte1998long} by replacing the driving  Brownian motion of the volatility process by a fractional Brownian motion  $W^H$:
$$\resizebox{0.97\textwidth}{!}{$ W_t^H = \frac{1}{\Gamma(H+1/2)}\int_0^t (t-s)^{H-1/2}dW_s + \frac{1}{\Gamma(H+1/2)}\int_{-\infty}^0 ((t-s)^{H-1/2}-(-s)^{H-1/2}) dW_s$}$$
where $H\in (0,1)$ is the Hurst exponent: $H>1/2$ corresponds to positively correlated returns, $H<1/2$ to negatively correlated increments and  $H=1/2$ reduces to the case of  standard Brownian motion. Sample paths of $W^H$ are locally H\"older continuous of any order strictly less than $H$, thereby less regular than standard Brownian motion. Initially \citet{comte1998long} considered the case $H>1/2$. However, a smaller Hurst index $H\approx 0.1$ allows to match exactly the regularity of the volatility time series and the exponent in the power--law decay of the at-the-money skew measured on the market (Figure~\ref{fig:skew intro}). Consequently models involving the fractional kernel $t\mapsto t^{H-1/2}$ with $H<1/2$ have been dubbed  ``rough volatility models" by \citet{gatheral2018volatility}.

The price to pay is that, in general, such models are no longer Markovian nor semimartingales, which limits their practical use and make their mathematical analysis quite challenging.  This  has initiated a thriving branch of research.\footnote{Refer to \url{https://sites.google.com/site/roughvol/home} for references.}
The need for fast pricing in such non-Markovian models is therefore, more than ever, crucial.  One breakthrough in that direction was achieved by  \citet{EER:06} who came up with a rough version of the \citet{heston1993closed} model  after convolving the dynamics \eqref{intro:heston} with a fractional kernel to get 
\begin{align}
V_t  &=  V_0 + \frac{1}{\Gamma(H+1/2)}\int_0^{t}  (t-s)^{H-1/2}\left((\theta -\kappa V_s) ds+ \nu \sqrt{V_s} dW_s\right), \label{eq:intro rheston2}
\end{align}
for $H \in (0,1/2)$.  Remarkably, they show that  an analogous formula for the characteristic function of the log price  to that of \citet{heston1993closed}    continue to hold modulo a fractional deterministic Riccati equation. From a theoretical perspective, the rough Heston model falls  into the broader class of non-Markovian affine   Volterra processes developed in \citet{AJLP17,abijaber2019weak}, and can be recovered as a projection of infinite dimensional Markovian affine processes as illustrated in \citet{abi2019markovian,cuchiero2018generalized,gatheral2019affine}.

Although  the rough Heston model can be efficiently implemented \citep{abi2019lifting,abi2019multifactor, callegaro2020fast, gatheral2019rational},  no closed-form solution for the fractional deterministic Riccati equation and whence for the characteristic function is known to date, which has to be contrasted with the conventional \citet{heston1993closed} model. One possible explanation could be that, unlike the Markovian case, squares of fractional Brownian motion have different dynamics than \eqref{eq:intro rheston2}, so that the marginals of the process  \eqref{eq:intro rheston2} are not chi-square distributed, except for the case $H=1/2$. 

The main objective of the paper is to rely on squares of general Gaussian processes with arbitrary covariance structures by considering the non-Markovian extension of the \citet{stein1991stock} and the \citet{schobel1999stochastic} models. We will show that the underlying Gaussianity makes the problem highly tractable and allows to recover analytic expressions for the joint Fourier--Laplace transform of the log price and the integrated variance in general, which would agree with that of Stein--Stein, Schöbel--Zhu and Heston under the Markovian setting.   Such models have been already considered several times in the context of non-Markovian and rough volatility literature   \citep{cuchiero2019markovian,gulisashvili2019extreme, HS18, horvath2017asymptotic} but there has been no derivation  of the analytic form of the characteristic function.   Our methodology takes a step further the recent derivation in  \citet{AJ2019laplace} for the Laplace transform of the integrated variance and that of \citet{abi2020markowitz} where the Laplace transform of the forward co-variance curve enters in the context of portfolio optimization.

\subsection*{The Gaussian Stein--Stein model and  main results}
For $T>0$, we  will consider the following generalized version of the Stein--Stein model:
\begin{align}
d S_t &= S_t X_t  dB_t, \quad S_0>0, \label{eq:steinsteinS}\\
X_t &= g_0(t) +\int_0^T K(t,s)\kappa X_s ds  + \int_0^T K(t,s) \nu dW_s,  \label{eq:steinsteinX} 
\end{align}
with $B= \rho W + \sqrt{1-\rho^2} W^{\perp}$,  $\rho \in [-1,1]$, $\kappa,\nu \in \R$, $g_0$ a suitable deterministic  input curve, $K:[0,T]^2\to \R$ a measurable kernel and $(W,W^{\perp})$ a two-dimensional Brownian motion.

Under mild assumptions on its covariance function, every Gaussian process   can be written in the form \eqref{eq:steinsteinX} with $\kappa=0$, see \citet{sottinen2016stochastic}. Such representation is known as the Fredholm representation.  We will be chiefly interested in two classes of kernels $K$:
\begin{itemize}
	\item 
Symmetric kernels, i.e. $K(t,s)=K(s,t)$ for all $s,t\leq T$, for which the integration in \eqref{eq:steinsteinX} goes up to time $T$, meaning that $X$ is not necessarily adapted to the filtration generated by $W$. {In this case, the stochastic integral $\int_0^{\cdot} X dB$ cannot be defined in a dynamical way as an It\^o integral whenever $\rho \neq 0$. We will make sense of \eqref{eq:steinsteinS}--\eqref{eq:steinsteinX} in a static sense in Section \ref{S:symm}.}
\item
Volterra kernels, i.e. $K(t,s)=0$ whenever $s\geq t$, for which integration in  \eqref{eq:steinsteinX} goes up to time $t$, which is more in line with standard stochastic volatility modeling {and for which the stochastic integral $\int_0^{\cdot} XdB$ can be defined in the  usual It\^o sense, see Section \ref{S:volterra}}. For instance, the conventional mean reverting Stein-Stein model \eqref{intro:ssvol} can be recovered by setting $g_0(t)=X_0 - \kappa \theta t$, $\kappa \leq 0$ and by considering the Volterra kernel   $K(t,s)=\bm {1}_{s< t}$.  The fractional Brownian motion with a Hurst index $H \in (0,1)$ can be represented using the Volterra kernel 	\begin{align*}
K(t,s)= \bm {1}_{s< t}\frac{(t-s)^{H-1/2}}{\Gamma(H+\frac 1 2)} \, {}_2 F_1\left(H-\frac 1 2; \frac 1 2-H; H+\frac 1 2; 1-\frac t s \right),
\end{align*}
where ${}_2F_1$ is the Gauss hypergeometric function; and the Riemman-Liouville fractional Brownian motion corresponds to the case $K(t,s)=\bm {1}_{s< t}(t-s)^{H-1/2}/\Gamma(H+1/2)$.
\end{itemize}

For suitable $u,w\in \mathbb C$, we provide the following analytical expression for the conditional joint Fourier--Laplace transform of the log-price  and the integrated variance:
\begin{align}\label{eq:introchar}
\!\!\!\!\!\!\!\!	\E\left[ \exp\left(u \log \frac{S_T}{S_t} + w \int_t^T X_s^2 ds  \right) \! \!\Mid \mathcal \! \mathcal F_t\right] =  \frac{ \exp\left(\langle g_t, \boldsymbol{\Psi}_{t} g_t \rangle_{L^2} \right)}{\det\left( \boldsymbol{\Phi}_t \right) ^{1/2}},
\end{align}
with   $\langle f,h\rangle_{L^2}=\int_0^Tf(s)h(s)ds$, $\det$ the \citet{fredholm1903} determinant (see Appendix~\ref{A:trace}),  $g_t$ the adjusted conditional mean given by
\begin{align}\label{eq:condmean}
g_t(s)= \bm 1_{t\leq s} \E\left[X_s -\int_t^T K(s,r)\kappa X_r dr \Mid \Fc_t \right], \quad s,t\leq T;
\end{align}
and $\boldsymbol{\Psi}_{t}$ a linear operator  
 acting on $L^2\left([0,T],\R\right)$ defined by
	\begin{align}\label{def:riccati_operator}
\boldsymbol \Psi_{t}= \left( \id - b \boldsymbol{K}^* \right)^{-1}a \left( \id - 2 a \boldsymbol{\tilde \Sigma}_{t}  \right)^{-1}     \left( \id - b  \boldsymbol{K} \right)^{-1},  \quad  t\leq T, 
\end{align}
{where $\boldsymbol K$ denotes the integral operator induced by $K$, $\boldsymbol{K}^*$ the adjoint operator,\footnote{cf. below for detailed notations.}}   $\id$ denotes the identity operator, i.e. $(\id f)=f$ for all $f \in L^2\left([0,T],\C \right)$, 
\begin{align}\label{eq:ab}
 a = w + \frac 1 2 (u^2-u), \quad b=\kappa +\rho\nu u,
\end{align}
 and $\boldsymbol{\tilde{\Sigma}}_t$ the adjusted covariance integral operator defined by
\begin{align}\label{def:C_tilde}
	\boldsymbol{\tilde{\Sigma}}_t = (\id- b\boldsymbol{K})^{-1} \boldsymbol{\Sigma}_t (\id-b \boldsymbol{K}^*)^{-1},
\end{align}
with $\boldsymbol{\Sigma}_t$ defined as the integral operator associated with the covariance kernel
\begin{align}\label{eq:sigmakernel}
	{\Sigma}_t(s,u) =\nu^2  \int_t^{T} K(s,z)K(u,z)  dz, \quad t \leq s,u \leq T,
\end{align}
and finally $\boldsymbol{\Phi}$ is defined by  
  \begin{equation}
\boldsymbol{\Phi}_t =
\begin{cases*}
(\id - b\boldsymbol{K})(\id - 2 a\boldsymbol{\tilde \Sigma}_{t})(\id-b\boldsymbol{K}) & if $K$ is a symmetric kernel \\
\id - 2 a\boldsymbol{\tilde \Sigma}_{t}     & if $K$ is a Volterra kernel
\end{cases*}.
\end{equation}

At first glance, the expressions for $\boldsymbol{\Phi}$ seem to depend on the class of the kernel, but they actually agree. Indeed, for Volterra kernels, i.e. $K(t,s)=0$ for $s\geq t$, $\det(\id -b\boldsymbol{K})=\det(\id -b\boldsymbol{K}^*)=1$ so that  using the relation \citep[Theorem 3.8]{simon1977notes} $\det((\id +\boldsymbol{F})(\id+\boldsymbol{G}))=\det(\id +\boldsymbol{F})\det(\id +\boldsymbol{G})$:
$$  \det((\id - b\boldsymbol{K})(\id - 2 a \boldsymbol{\tilde \Sigma}_{t})(\id-b\boldsymbol{K}^*))=\det(\id - 2 a\boldsymbol{\tilde \Sigma}_{t}).$$

As already mentioned, we prove \eqref{eq:introchar} for two classes of kernels:
\begin{itemize}
	\item 
\textbf{Symmetric nonnegative kernels}: we  provide an elementary static derivation  of   \eqref{eq:introchar} for $t=0$ and $\kappa=0$, based on the spectral decomposition of $K$ which  leads to the  decomposition of the characteristic function as an infinite product
	of independent Wishart distributions. The operator $\boldsymbol{\Psi}_0$  appears naturally after a rearrangement of the terms.  The main result is collected in  Theorem~\ref{T:symm}.
	\item
	\textbf{Volterra kernels:} {under some $L^2$-continuity and boundedness condition,}   we adopt a dynamical approach to derive the conditional characteristic function \eqref{eq:introchar} via It\^o's formula on the adjusted conditional mean process $(g_t(s))_{t\leq s}$. The main result is stated in Theorem~\ref{T:volterra}. This is the class of kernels which is more suited for financial applications. 
\end{itemize}

From the numerical perspective, we will show in Section~\ref{S:numeric} that the expression \eqref{eq:introchar}  lends itself to approximation by closed form solutions   using finite dimensional matrices after a straightforward discretization of the operators in the form 
\begin{align}
\!\!\!\!\!\!\!\!	\E\left[ \exp\left(u \log \frac {S_T} {S_0} + w \int_0^T X_s^2 ds  \right) \right] \approx \frac{\exp\left( \frac{T}{n} g_n^\top {\Psi}_0^n g_n  \right)}{\det(\Phi^n_0)^{1/2}}
\end{align}
where $g_n\in \R^{n}$ and $\Phi^n_0,  {\Psi}_0^n \in \R^{n\times n}$ are entirely determined by $(g_0,K,\nu,\kappa,u,w)$ and $\det$ is the standard determinant of a matrix, we refer to Section~\ref{S:numeric}. We illustrate the applicability of these formulas on an option pricing and calibration example by Fourier inversion techniques in a (rough) fractional Stein--Stein model in Section~\ref{S:pricing}.

\subsection*{Notations}
 Fix $T>0$. We let $\mathbb K$ denote $\R$ or $\mathbb C$. We denote by $\langle \cdot, \cdot \rangle_{L^2}$ the following product 
\begin{align}
\langle f, g\rangle_{L^2} = \int_0^T f(s)^{\top} g(s) ds, \quad f,g\in L^2\left([0,T],\mathbb K\right). 
\end{align}
We note that $\langle \cdot, \cdot \rangle_{L^2}$ is an inner product on $L^2\left([0,T],\mathbb R\right)$, but not on $L^2\left([0,T],\mathbb C\right)$.
We define $L^2\left([0,T]^2,\mathbb K\right)$ to be the space of  measurable kernels $K:[0,T]^2 \to \mathbb K$ such that 
\begin{align*}
\int_0^T \int_0^T |K(t,s)|^2 dt ds < \infty.
\end{align*}
For any $K,L \in  L^2\left([0,T]^2,\mathbb K\right)$ we define the $\star$-product   by
\begin{align}\label{eq:starproduct}
(K \star L)(s,u) = \int_0^T K(s,z) L(z,u)dz, \quad  (s,u) \in [0,T]^2,
\end{align}
which is well-defined in $L^2\left([0,T]^2,\mathbb K\right)$ due to the Cauchy-Schwarz inequality.  For any  kernel $K \in L^2\left([0,T]^2,\mathbb K\right)$, we denote by {$\boldsymbol K$} the integral operator   induced by the kernel $K$ that is 
\begin{align}
({\boldsymbol K} g)(s)=\int_0^T K(s,u) g(u)du,\quad g \in L^2\left([0,T],\mathbb K\right).
\end{align}
$\boldsymbol K$ is a linear bounded operator from  $L^2\left([0,T],\mathbb K\right)$ into itself. 
If $\boldsymbol{K}$ and $\boldsymbol{L}$ are two integral operators induced by the kernels $K$ and $L$  in $L^2\left([0,T]^2,\mathbb K\right)$, then $\boldsymbol{K}\boldsymbol{L}$ is {also an} integral operator induced by the kernel $K\star L$.

When $\mathbb K=\R$, we denote by $K^*$ the adjoint kernel of $K$ for $\langle \cdot, \cdot \rangle_{L^2}$, that is 
\begin{align}
K^*(s,u) &= \; K(u,s), \quad  (s,u) \in [0,T]^2,
\end{align}
and by $\boldsymbol{K}^*$ the corresponding adjoint integral operator. 

{The square-root of a complex number $\sqrt{z}$ is defined through its main branch,~i.e.~$\sqrt{z}=|z|e^{i\mbox{arg}(z)/2}$  with $z=|z|e^{i\mbox{arg}(z)}$ such that $\arg(z) \in (-\pi, \pi]$.}
\vspace{1mm}

\section{Symmetric kernels: {an elementary static approach}}\label{S:symm}
We provide an elementary static derivation of the joint Fourier--Laplace transform in the special case of symmetric kernels with $\kappa=0$. We stress that, although the case of symmetric kernels is not of interest for practical applications, it naturally leads through direct computations to the analytic expression \eqref{eq:introchar} in terms of the operator $\boldsymbol{\Psi}$ given in \eqref{def:riccati_operator}. Later on, in Section~\ref{S:volterra}, such expressions are shown to hold in the more practical case of Volterra kernels using a dynamical approach.

{\begin{definition}\label{D:nonnegative}
	A linear operator  $\boldsymbol K$  from  $L^2\left([0,T],\mathbb R\right)$ into itself is symmetric nonnegative if $\boldsymbol K=\boldsymbol K^*$ and   $\langle f,\boldsymbol{K}f\rangle_{L^2}\geq 0$, for all $f  \in L^2\left([0,T],\R\right)$.   Whenever	$\boldsymbol K$ is an integral operator induced by some kernel  $K \in L^2\left([0,T]^2,\R\right)$, we will say that $K$ is symmetric nonnegative. In this case,  it follows that  $K=K^*$ a.e. and  
	$$   \int_0^T \int_0^T f(s)^\top K(s,u) f(u) du ds \geq 0, \quad  \forall f  \in L^2\left([0,T],\R\right).  $$
	 $\boldsymbol{K}$ is said to be symmetric nonpositive, if $(-\boldsymbol{K})$ is symmetric nonnegative. 
\end{definition}}

Throughout this section, we fix $T>0$  and we consider the case of symmetric kernels having the following spectral decomposition 
\begin{align}\label{eq:decompK}
K(t,s)=\sum_{n \geq 1} \sqrt{\lambda_n} e_n(t)e_n(s), \quad t,s\leq T,
\end{align}
where $(e_n)_{n\geq 1}$ is an orthonormal basis of $L^2([0,T],\R)$ for  the inner product $\langle f,g \rangle_{L^2}=\int_0^T f(s)g(s)ds$ and  $\lambda_1 \geq \lambda_2 \geq \ldots \geq 0$ with $\lambda_n\to 0$, as $n\to \infty$, such that 
\begin{align}\label{eq:sumlamb}
\sum_{n \geq 1} \lambda_n <\infty. 
\end{align}
Such decomposition is possible whenever the operator $\boldsymbol{K}$ is the (nonnegative symmetric) square-root of a covariance operator $\boldsymbol{C}$ which is generated by a continuous kernel.  This is known as Mercer's theorem, see \citet[Theorem 1, p.208]{shorack2009empirical} and leads to the so-called Kac--Siegert/Karhunen--Lo\`eve representation of the process $X$, see \citet{kac1947theory, karhunen1946spektraltheorie, loeve1955probability}. In this case, one can show that any square-integrable Gaussian process $X$ with mean  $g_0$ and covariance $\boldsymbol{C}$ admits the representation \eqref{eq:steinsteinX} with $\kappa=0$  on some filtered probability space supporting a Brownian motion $W$, see \citet[Theorem 12]{sottinen2016stochastic}.
	
{We start by making precise how one should understand \eqref{eq:steinsteinS}--\eqref{eq:steinsteinX} in the case of symmetric kernels and $\kappa = 0$. We rewrite \eqref{eq:steinsteinS} in the equivalent form 
	\begin{align}\label{eq:logSsym}
	\log S_t = \log S_0 - \frac 1 2 \int_0^t X_s^2 ds + \rho  \int_0^t X_s dW_s + \sqrt{1-\rho^2} \int_0^t X_s dW_s^{\perp}.
	\end{align}  
	We fix $T>0$, $g_0 \in L^2([0,T], \R)$  and a complete probability space $(\Omega, \mathcal F, \mathbb Q)$ supporting a two dimensional Brownian motion $(W,W^{\perp})$ and, for each $t \leq T$, we set 
\begin{align}\label{eq:Xsym}
X_t = g_0(t) + \int_0^T K(t,s) \nu dW_s.
\end{align}
We note that \eqref{eq:decompK}--\eqref{eq:sumlamb} imply that $K \in L^2([0,T]^2, \R)$ so that the stochastic integral $\int_0^T K(t,s)\nu dW_s$ is well-defined as an Itô integral for almost every $t\leq T$  and $X$ has sample paths in $L^2([0,T],\R)$ almost surely.  Setting $\mathcal F_t =\mathcal F^X_t \vee \mathcal F^{W^{\perp}}_t$ where $(\mathcal F^Y_t)_{t \geq 0}$ stands for the filtration generated by the process $Y$, we have that $W^{\perp}$ is still a Brownian motion with respect to $(\mathcal F_t)_{t\geq 0}$ {and, up to a modification, $X$ is progressively measurable\footnote{Every jointly measurable and adapted process admits a progressively measurable modification, see \cite{ondrejat2013existence}.} with respect to $(\mathcal F^X_t)_{t\geq 0}$ (and whence w.r.t.~the enlarged filtration $(\mathcal F_t)_{t\geq 0}$)} so that  
$$ \int_0^{\cdot} X_s dW^{\perp}_s $$ is well defined as an Itô integral with respect to $(\mathcal F_t)_{t\geq 0}$. If $\rho = 0$, \eqref{eq:logSsym} is therefore well-defined in the classical way. 
 However, for $\rho \neq 0$,   since $X$ is not necessarily adapted to the filtration generated by $W$ (and vice versa), $W$ is no longer necessarily a Brownian motion with respect to the extended filtration $\mathcal F^X \vee \mathcal F^W$, and one cannot make sense of the  stochastic integral $\int_0^{\cdot} X dW$  in the usual dynamical sense. We provide a static interpretation of \eqref{eq:logSsym} valid only at the terminal time $T$.  To this end, since $g_0 \in L^2([0,T],\R)$, we can write $g_0=\sum_{n\geq 1} \langle g_0,e_n  \rangle e_n $. Making use of \eqref{eq:decompK}, we first observe  that, an application of Fubini's theorem \cite[Theorem 2.2]{veraar2012stochastic}, justified by the fact that 
 $$\int_0^T\sum_{n\geq 1}\E\left[\int_0^T |\sqrt{\lambda_n} e_n(t) e_n(s)|^2 ds  \right] dt = \sum_{n\geq 1} \lambda_n \int_0^T e_n(t)^2 dt\leq \sum_{n\geq 1} \lambda_n<\infty, $$
 yields that
\begin{align}\label{eq:repX}
X_t = g_0(t) + \int_0^T K(t,s) \nu dW_s =\sum_{n\geq 1} \left( \langle g_0,e_n  \rangle + \sqrt{\lambda_n} \nu \xi_n   \right)e_n(t), \quad dt \otimes \mathbb Q-a.e.
\end{align} 
where $\xi_n=\int_0^T e_n(s)dW_s$, for each $n\geq 1$. Since $(e_n)_{n\geq 1}$ is an orthonormal family in $L^2$, $(\xi_n)_{n\geq 1}$ is a sequence of independent standard Gaussian random variables that are $\mathcal F^W_T$ measurable. 
We set
\begin{align}\label{eq:repX3}
N_T = \sum_{n\geq 1} \left( \langle g_0,e_n  \rangle + \sqrt{\lambda_n} \nu \xi_n   \right) \xi_n.
\end{align}

\begin{remark}
We note that $N_T$ plays the role of $\int_0^T X_s dW_s$, since a formal interchange leads to 
\begin{align}
N_T & = \sum_{n\geq 1} \left( \langle g_0,e_n  \rangle + \sqrt{\lambda_n} \nu \xi_n   \right) \int_0^T e_n(s)dW_s \\
&(=)\int_0^T \sum_{n\geq 1} \left( \langle g_0,e_n  \rangle + \sqrt{\lambda_n} \nu \xi_n   \right)  e_n(s)dW_s\\
&(=) \int_0^T X_s dW_s.   
\end{align}
Obviously, since $\xi_n$ are not adapted the integral $\int_0^{\cdot} \xi_n e_n(s)dW_s$ cannot be defined in the non-anticipative sense.	
\end{remark}
  Finally, we take as definition for  the log-price at the terminal time  $T$:  
\begin{align}\label{eq:logSsym2}
\log S_T = \log S_0 -\frac 1 2 \int_0^T X_s^2 ds + \rho N_T + \sqrt{1-\rho^2}\int_0^T X_s dW_s^{\perp}, \quad S_0>0,
\end{align}
 which is an $\mathcal F^W_T \vee \mathcal F_T$-measurable random variable. 
}

We state our main result of the section on the representation of the characteristic function for symmetric kernels. 
\begin{theorem}\label{T:symm}
	Let $K$ be as in \eqref{eq:decompK}, $g_0 \in L^2([0,T],\R)$ and set $\kappa=0$.
	Fix $u,w \in \mathbb C$ such that $\Re(u)=0$ and  $\Re(w)\leq 0$. Then,
	\begin{align}\label{eq:charfuncsym}
	\E\left[ \exp \left( u\log \frac{S_T}{S_0} +  w\int_0^T X_s^2 ds \right)  \right] =\frac{\exp\left( {\langle g_0 , \boldsymbol{\Psi}_0 g_0 \rangle_{L^2} } \right)}{\det\left(\boldsymbol{\Phi}_0  \right)^{1/2}},
	\end{align}
	with $\boldsymbol{\Psi}_0$ and $\boldsymbol{\tilde{\Sigma}}_0$ respectively  given by \eqref{def:riccati_operator} and \eqref{def:C_tilde},  for $(a,b)$ as in \eqref{eq:ab} (with $\kappa=0$), that is 
	\begin{align}
	a = w + \frac 1 2 (u^2-u), \quad b=\rho\nu u,
	\end{align}
	and 
	$\Phi_0 = (\id - b \boldsymbol K)(\id - 2 a\boldsymbol{\tilde \Sigma}_{0})(\id-b \boldsymbol K).$
\end{theorem}

The rest of the section is dedicated to the proof of Theorem~\ref{T:symm}. The key idea is to rely on the spectral decomposition \eqref{eq:decompK} to  decompose the characteristic function as an infinite product
of independent Wishart distributions. The operators $\boldsymbol{\tilde{\Sigma}}_0$ and $\boldsymbol{\Psi}_0$ will then appear naturally after a rearrangement of the terms. 

In the sequel, to ease notations, we drop the subscript $L^2$ in the product $\langle \cdot, \cdot \rangle_{L^2}$.
We will start by computing the joint Fourier--Laplace transform of $\left(\int_0^{T} X_s^2 ds, N_T\right)$. Furthermore, the representation \eqref{eq:repX} readily leads to 
\begin{align}
\int_0^T X_s^2 ds &= \sum_{n\geq 1} \left( \langle g_0,e_n  \rangle + \sqrt{\lambda_n} \nu \xi_n   \right)^2. \label{eq:repX2} 
\end{align}

\begin{lemma}\label{L:X2XdW} 	 Let $K$ be as in \eqref{eq:decompK}, $g_0 \in L^2([0,T],\R)$, set $\kappa=0$ and fix {$\alpha,\beta \in \mathbb C$} such that 
	\begin{align}\label{eq:alphabeta}
	\Re(\alpha)\leq 0 , \quad \Re(\beta)=0.
	\end{align}
	Then,
	\begin{align}\label{eq:laplaceX2XdW}
	\E\left[ \exp \left( \alpha \int_0^T X_s^2 ds +  \beta {N_T} \right)  \right] = \frac{\exp\left( \left(\alpha + \frac{\beta^2}{2}\right)\sum_{n \geq 1}  \frac{\langle g_0,e_n \rangle^2}{1-2\beta\nu \sqrt{\lambda_n} - 2 \alpha \nu^2 \lambda_n} \right)}{\prod_{n \geq 1} \sqrt{1-2\beta\nu \sqrt{\lambda_n} - 2 \alpha \nu^2 \lambda_n}}.
	\end{align}
\end{lemma}

\begin{proof}
	Define $U_T=\alpha \int_0^T X_s^2 ds +  \beta  {N_T}$.
	We first observe that \eqref{eq:alphabeta} yields that $\left|\exp \left( U_T \right)\right|=\exp(\Re(U_T)) \leq 1$, so that $\E\left[ \exp \left( U_T \right)  \right]$ is finite.  By virtue of the representations \eqref{eq:repX3} and \eqref{eq:repX2}, we have 
	\begin{align}
	U_T = \sum_{n\geq 1} \alpha \tilde \xi_n^2 + \beta \tilde \xi_n \xi_n,
	\end{align}
	where $\tilde \xi_n = \left( \langle g_0, e_n \rangle + \nu \sqrt{\lambda_n} \xi_n \right)$,  for each $n\geq 1$. Setting $Y_n =(\tilde \xi_n,\xi_n)^\top  $, it follows that $(Y_n)_{n\geq 1}$ are independent such that each $Y_n$ is a two dimensional Gaussian vector with mean $\mu_n$  and covariance matrix  $\Sigma_n$ given by 
	\begin{align}\label{eq:munsigman}
\mu_n=	\begin{pmatrix}
	\langle g_0,e_n\rangle\\
	0
	\end{pmatrix} \quad  \mbox{and} \quad \Sigma_n	 = \begin{pmatrix}
	\nu^2 \lambda_n & \nu \sqrt{\lambda_n}\\
	\nu \sqrt{\lambda_n} & 1 
	\end{pmatrix}.
	\end{align} 
	Furthermore, we have 
	$$  U_T = \sum_{n\geq 1} Y_n^\top   w_n Y_n,  $$
	with 
	$$w_n	 = \begin{pmatrix}
	\alpha & \frac{\beta}{2}\\
	\frac{\beta}{2} & 0 
	\end{pmatrix}.$$
	By successively using the independence of $Y_n$ and the well-known expression for the characteristic function of the Wishart distribution, see for instance \citet[Proposition A.1]{AJ2019laplace}, we get
	\begin{align}
	\E\left[  \exp(U_T)\right] &= 	\E\left[  \exp\left( \sum_{n\geq 1} Y_n^\top   w_n Y_n\right)\right] \\
	&= \prod_{n \geq 1} \E\left[\exp\left( Y_n^\top  w_n Y_n\right)\right]\\
	&= \prod_{n \geq 1} \frac{\exp\left( \tr\left( w_n \left(I_2 - 2\Sigma_nw_n \right)^{-1} \mu_n \mu_n^\top  \right)\right)}{\det\left(I_2 - 2\Sigma_nw_n \right)^{1/2}} .
	\end{align}
	We now compute the right hand side. We have
	$$ (I_2-2\Sigma_n w_n)	 = \begin{pmatrix}
1-2\alpha	\nu^2 \lambda_n -\beta  \nu \sqrt{\lambda_n} &  -\beta \nu^2 \lambda_n\\
	-2\alpha \nu \sqrt{\lambda_n} -\beta &1-\beta\nu \sqrt{\lambda_n}  
	\end{pmatrix}$$
	so that 
	$$ \det (I_2-2\Sigma_n w_n) = 1-2\beta\nu \sqrt{\lambda_n} - 2 \alpha \nu^2 \lambda_n $$
	and 
	$$ (I_2-2\Sigma_n w_n)^{-1}= \frac{1}{1-2\beta\nu \sqrt{\lambda_n} - 2 \alpha \nu^2 \lambda_n} \begin{pmatrix}
	1-\beta\nu \sqrt{\lambda_n} &  \beta \nu^2 \lambda_n\\
	2\alpha \nu \sqrt{\lambda_n} +\beta   &1-2\alpha	\nu^2 \lambda_n -\beta  \nu \sqrt{\lambda_n}
	\end{pmatrix}.$$
	Straightforward computations lead to the claimed expression \eqref{eq:laplaceX2XdW}. 
\end{proof}

Relying on the spectral decomposition \eqref{eq:decompK}, we re-express the quantities entering in \eqref{eq:laplaceX2XdW} in terms of  suitable operators.  
\begin{lemma}\label{L:symspectral}
	 Let $K$ be as in \eqref{eq:decompK}, set $\kappa=0$ and fix {$\alpha,\beta \in \mathbb C$} as in \eqref{eq:alphabeta}.
	Then,  the following operator defined by \eqref{def:riccati_operator} with $a=\alpha + \frac{\beta^2}2$ and $b=\nu\beta$:
		\begin{align}\label{def:riccati_operatoren0}
	\boldsymbol{\Psi}^{\alpha,\beta}_{0}= \left( \id - b \boldsymbol{K}^* \right)^{-1} a\left( \id - 2 \boldsymbol{\tilde \Sigma}_{0}a  \right)^{-1}     \left( \id - b \boldsymbol{K} \right)^{-1},  \quad  t\leq T, 
	\end{align}
	admits the following decomposition
	\begin{align}\label{eq:determinant0}
		\boldsymbol{\Psi}^{\alpha,\beta}_{0} = \sum_{n\geq 1} \frac{\alpha + \frac{\beta^2}{2}}{1-2\beta\nu \sqrt{\lambda_n} - 2 \alpha \nu^2 \lambda_n} \langle e_n, \boldsymbol{\cdot}\, \rangle  e_n
	\end{align}
	and 
	\begin{align}\label{eq:determinant1}
	\det \left( \frac 1 {\alpha + \frac{\beta^2}2} \boldsymbol{\Psi}_0^{\alpha,\beta}  \right)={\prod_{n \geq 1} \frac{1}{{1-2\beta\nu \sqrt{\lambda_n} - 2 \alpha \nu^2 \lambda_n}}},
	\end{align}
	with the convention that $0/0=1$. In particular, 
		\begin{align}\label{eq:laplaceX2XdW2}
\!	\E\left[ \exp \left( \alpha \int_0^T X_s^2 ds +  \beta {N_T} \right)  \right] = 	\det \left( \frac 1 {\alpha + \frac{\beta^2}2} \boldsymbol{\Psi}_0^{\alpha,\beta}  \right)^{1/2}{\exp\left( \langle g_0, \boldsymbol{\Psi}^{\alpha,\beta}_{0} g_0 \rangle \right)}.\,\,\;\;
	\end{align}
\end{lemma}

\begin{proof} {Throughout the proof, we will make use of the following rule for computing the decomposition of a product of operators in terms of the orthonormal basis $(e_n)_{n \geq 1}$: for $\boldsymbol{K}$ and $\boldsymbol{L}$ in the form
$$ \boldsymbol{K} = \sum_{n \geq 1} a_n \langle e_n, \boldsymbol{\cdot}\, \rangle  e_n , \quad  \boldsymbol{L} = \sum_{n \geq 1} b_n \langle e_n, \boldsymbol{\cdot}\, \rangle  e_n , $$	
the composition is given by
$$ \boldsymbol{K} \boldsymbol{L} = \sum_{n \geq 1} a_n \langle e_n, \sum_{m \geq 1} b_m   e_m  \langle e_n, \boldsymbol{\cdot}\, \rangle  \rangle  e_n = \sum_{n\geq 1} a_n b_n \langle e_n, \boldsymbol{\cdot}\, \rangle  e_n. $$	
}
It follows from \eqref{eq:decompK} that 
	\begin{align}
	(\id -b\boldsymbol K) =\sum_{n\geq 1} \left(1-b\sqrt{\lambda_n}\right)   \langle e_n, \boldsymbol{\cdot}\, \rangle e_n.
	\end{align}
Since $\Re(\beta)=0$,  $\Re(1- b \sqrt{\lambda_n})=1\neq 0$ for each $n\geq 1$, so that $(\id -b\boldsymbol K)$ is invertible with an inverse given by 
	\begin{align}\label{tempidK-1}
(\id -b\boldsymbol K)^{-1} =\sum_{n\geq 1} \frac 1 { 1-b\sqrt{\lambda_n}}   \langle e_n, \boldsymbol{\cdot}\, \rangle e_n.
\end{align}
Similarly,  recalling \eqref{eq:sigmakernel}, \eqref{eq:decompK} leads to the representation of $\boldsymbol{\Sigma}_0=\nu^2 \boldsymbol{K} \boldsymbol{K}^*$:
\begin{align}
\boldsymbol{\Sigma}_0 = \sum_{n\geq 1} \nu^2 \lambda_n   \langle e_n, \boldsymbol{\cdot}\, \rangle e_n,
\end{align}
so that $\boldsymbol{\tilde \Sigma}_{0}$ given by \eqref{def:C_tilde} reads
\begin{align}
\boldsymbol{\tilde \Sigma}_0 = \sum_{n\geq 1} \frac { \nu^2 \lambda_n } { \left(1-b\sqrt{\lambda_n}\right)^2}    \langle e_n, \boldsymbol{\cdot}\, \rangle e_n.
\end{align}
Whence, 
\begin{align}
\left( \id - 2 a \boldsymbol{\tilde \Sigma}_{0} \right) = \sum_{n\geq 1} \frac {\left(1-b\sqrt{\lambda_n}\right)^2- 2a\nu^2 \lambda_n } { \left(1-b\sqrt{\lambda_n}\right)^2}    \langle e_n, \boldsymbol{\cdot}\, \rangle e_n.
\end{align}
Recalling that $a=\alpha + \frac{\beta^2}2$ and $b=\nu\beta$, 
 $\left(\left(1-b\sqrt{\lambda_n}\right)^2- 2a\nu^2 \lambda_n\right) = 1- 2 \nu \beta \sqrt{\lambda_n} - 2\alpha \nu^2 \lambda_n $. Since $\Re(\alpha)\leq 0$ and $\Re(\beta)=0$, we have that $\Re(1- 2 \nu \beta \sqrt{\lambda_n} - 2\alpha \nu^2 \lambda_n)>0$  so that $\left( \id - 2 a\boldsymbol{\tilde \Sigma}_{0}  \right)$ is invertible with an inverse given by 
\begin{align}
\left( \id - 2 a \boldsymbol{\tilde \Sigma}_{0}  \right)^{-1} = \sum_{n\geq 1} \frac  { \left(1-\nu \beta \sqrt{\lambda_n}\right)^2} { 1- 2 \nu \beta \sqrt{\lambda_n} - 2\alpha \nu^2 \lambda_n  }   \langle e_n, \boldsymbol{\cdot}\, \rangle e_n.
\end{align}
The representations \eqref{eq:determinant0}-\eqref{eq:determinant1} readily follows after composing by $(\id-b\boldsymbol{K}^*)^{-1}a$ from the left, by $(\id-b\boldsymbol{K})^{-1}$ from the right and recalling  \eqref{tempidK-1}. Finally, combining these expressions with \eqref{eq:laplaceX2XdW}, we obtain \eqref{eq:laplaceX2XdW2}. This ends the proof. 
\end{proof}

We can now complete the proof of Theorem~\ref{T:symm}.
 
\begin{proof}[Proof of Theorem~\ref{T:symm}]
		It suffices to prove that
	\begin{align}\label{eq:tempchar}
	\E\left[ \exp \left( u\log \frac{S_T}{S_0} +  w\int_0^T X_s^2 ds \right)  \right] = \E\left[ \exp \left( \alpha \int_0^T X_s^2 ds +  \beta {N_T}\right)  \right],
	\end{align}
	where 
\begin{align}
\alpha  = w + \frac{1}{2}(u^2-u) - \frac{\rho^2u^2}{2} \quad \mbox{and} \quad \beta= \rho u. 
\end{align}
	Indeed, if this the case, then
$$	\Re(\alpha)= \Re(w) + \frac 1 2 (\rho^2-1)\Im(u)^2  \leq 0,$$
	 so that an application of Lemma~\ref{L:symspectral}  yields the expression \eqref{eq:charfuncsym}. \\
It remains to prove \eqref{eq:tempchar} by means of a projection argument. 
{Conditional on  $(\mathcal F^{X}_t\vee \mathcal F^W_t)_{t\leq T}$,  by independence of $X$ and $W^{\perp}$,  the random variable $\int_0^T X_s dW_s^{\perp}$ is centered gaussian with variance $\int_0^T X_s^2 ds$ so that}
\begin{align}
M_T:&=\E\left[ \exp\left( u\sqrt{1-\rho^2}\int_0^T X_sdW^{\perp}_s\right)\Mid\! (\mathcal F^{X}_t\vee \mathcal F^W_t)_{t\leq T}\right] \\
&=\exp\left( \frac {u^2(1-\rho^2)} 2  \int_0^T X_s^2 ds  \right) \label{eq:tempM}.
\end{align}
A successive application of  the tower property of the conditional expectation on the  expression  \eqref{eq:logSsym2}  yields that
\begin{align*}
	\E\left[ \exp \left( u\log \frac{S_T}{S_0} +  w\int_0^T X_s^2 ds \right)  \right] &=  \E\left[ \E\left[\exp \left( u\log \frac{S_T}{S_0} +  w\int_0^T X_s^2 ds \right) \Mid (\mathcal F^{X}_t\vee \mathcal F^W_t)_{t\leq T} \right] \right] \\
	&=	\EE\left[ \exp\left(  \left(w-\frac u 2  \right)\int_0^T X_s^2 ds + \rho u \int_0^T X_s dW_s  \right)  M_T\right]
\end{align*}
leading to \eqref{eq:tempchar} due to \eqref{eq:tempM}. This ends the proof.
\end{proof}
	
\section{Volterra kernels: a dynamical approach}\label{S:volterra}

In this section, we treat the class of Volterra kernels which are practically relevant in mathematical finance. We will consider 
the class of Volterra kernels of continuous and  bounded type  in $L^2$ in  the terminology of \citet[Definitions 9.2.1, 9.5.1 and 9.5.2]{GLS:90}. 

\begin{definition}\label{D:kernelvolterra} A kernel $K:[0,T]^2 \to \R$  is a Volterra kernel of continuous and  bounded type  in $L^2$ if   $K(t,s)=0$ whenever $s \geq t$ and
	\begin{align}
	\label{assumption:K_stein}
	\sup_{t\in [0,T]}\int_0^T |K(t,s)|^2 ds <  \infty, \quad  \lim_{h \to 0} \int_0^T |K(u+h,s)-K(u,s)|^2 ds=0, \quad u \leq T.
	\end{align}
\end{definition}

The following kernels are   of continuous and  bounded type  in $L^2$.

\begin{example}
	\begin{enumerate}
				\item 
		Any convolution kernel of the form $K(t,s)=k(t-s)\bold 1_{s< t}$	with $k\in L^2([0,T],\R)$. {Indeed, 
		$$ \sup_{t\leq T} \int_0^T |K(t,s)|^2 ds  = \sup_{t\leq T} \int_0^t |k(t-s)|^2 ds \leq   \int_0^T |k(s)|^2 ds<\infty$$ 
	yielding the first part of \eqref{assumption:K_stein}. The second part follows from the  $L^2$-continuity of $k$, see \cite[Lemma 4.3]{brezis2010functional}.}
		\item 
		For $H\in (0,1)$, 
			\begin{align*}
		K(t,s)= \bm {1}_{s< t}\frac{(t-s)^{H-1/2}}{\Gamma(H+\frac 1 2)} \, {}_2 F_1\left(H-\frac 1 2; \frac 1 2-H; H+\frac 1 2; 1-\frac t s \right),
		\end{align*}
		where ${}_2F_1$ is the Gauss hypergeometric function. Such kernel enters in the Volterra representation  \eqref{eq:steinsteinX} of   the fractional Brownian motion  whose covariance function is $\Sigma_0(s,u)= (K\star K^*)(s,u)=\frac {1} 2 (s^{2H}+u^{2H}-|s-u|^{2H})$, see  \citet{decreusefond1999stochastic}. {In this case, 
			$$ \sup_{t\leq T} \int_0^T |K(t,s)|^2 ds = \sup_{t\leq T}  \Sigma_0(t,t) \leq T^{2H} $$
			and by developing the square
			\begin{align*}
	\int_0^T |K(u+h,s)-K(u,s)|^2 ds = \Sigma_0(u+h,u+h)- 2\Sigma_0(u+h,u) + \Sigma_0(u,u)	
			\end{align*}	
		which goes to $0$ as $h\to 0$.}
		\item
		Continuous kernels $K$ on $[0,T]^2$. This is the case for instance for the Brownian Bridge $W^{T_1}$ conditioned to be equal to $W^{T_1}_0$ at a time $T_1$: for all $T<T_1$, $W^{T_1}$ admits the Volterra representation \eqref{eq:steinsteinX} on $[0,T]$ with the continuous kernel 
		$ K(t,s)  = \bm {1}_{s< t} (T_1-t)/(T_1-s)$, for all $s,t\leq T$.
		\item 
		If $K_1$ an $K_2$ satisfy  \eqref{assumption:K_stein} then so does $K_1\star K_2$ by an application of  Cauchy-Schwarz inequality.
	\end{enumerate}
\end{example}

Throughout this section, we fix a probability space $(\Omega,\mathcal F, (\mathcal F_t)_{t\leq T} ,\Q)$ supporting a two dimensional Brownian motion $(W,W^{\perp})$ and we set $B=\rho W + \sqrt{1-\rho^2} W^{\perp}$. 
{For any  Volterra kernel $K$ of continuous and bounded type  in $L^2$, and any $g_0 \in L^2([0,T],\R)$, there exists a progressively measurable  $\R\times \R_+$-valued strong solution $(X,S)$ to \eqref{eq:steinsteinS}-\eqref{eq:steinsteinX}  such that 
	\begin{align}\label{eq:momentsX}
	\sup_{t\leq T} \E\left[ |X_t|^p \right] < \infty, \quad p \geq 1,
	\end{align}
we refer to Theorem~\ref{T:existenceappendix} below for the proof. {It follows in particular from \eqref{eq:momentsX} that $\int_0^T X_s^2 ds < \infty$ almost surely, so that $X$ has sample paths in $L^2([0,T],\R)$.}

We now state our main result on the representation of the Fourier--Laplace transform for Volterra kernels  under  the following additional assumption on the kernel:
\begin{align}\label{eq:assumptionkerneldiff1}
\sup_{t\leq T} \int_0^T |K(s,t)|^2 ds < \infty.
\end{align}

\begin{theorem}\label{T:volterra}
	Let  $g_0\in L^2([0,T],\R)$ and $K$ be a Volterra kernel as in Definition~\ref{D:kernelvolterra} satisfying \eqref{eq:assumptionkerneldiff1}.
	Fix $u,w\in \mathbb C$, such that $0\leq \Re(u) \leq 1$ and $\Re(w)\leq0$.
	Then, 
	\begin{align}\label{eq:charvolterra}
\!\!\!\!\!\!\!\!	\E\left[ \exp\left(u \log \frac{S_T}{S_t} + w \int_t^T X_s^2 ds  \right) \! \!\Mid \mathcal \! \mathcal F_t\right] = {\exp\left( \phi_t + \langle g_t, \boldsymbol{\Psi}_{t} g_t \rangle_{L^2} \right)},
	\end{align}
	for all $t\leq T$, with $\boldsymbol{\Psi}_t$ given by \eqref{def:riccati_operator} for $(a,b)$ as in \eqref{eq:ab} and 
{\begin{align}\label{eq:phii}
\phi_t =- \int_t^T \Tr(\boldsymbol{\Psi}_t \boldsymbol{\dot{\Sigma}}_t)dt,
\end{align}
	where $\dot{\boldsymbol{{\Sigma}}}_t$ is the strong derivative\footnote{See Lemma~\ref{L:Psi} below.} of $t\mapsto\bold{\Sigma}_t$ induced by the  kernel
\begin{align}
\dot{\Sigma}_t(s,u)=-\nu^2 K(s,t) K(u,t) , \quad a.e.
\end{align}
and $\Tr$ is the trace operator, see Appendix~\ref{A:trace}.}
\end{theorem}

\begin{proof}
	We refer to Appendix~\ref{A:proofvolterra}.
\end{proof}

{The following  remark establishes the link between $\phi$ and the Fredholm determinant. 
	\begin{remark}\label{rmk:det}  Assume $u,w$ are real. 
		We recall the definition 
$$\Phi_t =\id - 2 \boldsymbol{\tilde{\Sigma}}_t a, \quad t\leq T,$$
and that $\boldsymbol{\tilde{\Sigma}}_t$ is an integral operator of trace class with continuous kernel by virtue of Lemma~\ref{L:Kt} below
so that  the determinant $\det(\Phi_t)$ is well defined and non-zero  by the invertibility of 
$(\id- 2 \boldsymbol{\tilde{\Sigma}}_t a )$, see Lemma~\ref{eq:Psiwelldefined} and \citet[Theorem 3.9]{simon1977notes}.
We set
\begin{align}\label{eq:detphi}
\phi_t = \log(\det(\Phi_t)^{-1/2})=-\frac 1 2 \log(\det(\Phi_t)).
\end{align}
Differentiation using the logarithmic derivative of the Fredholm's determinant (see \cite[Chap IV, p.158 (1.3)]{gohberg1978introduction}) and \eqref{def:C_tilde} yields 
\begin{align}
\dot {\phi}_t  =  \Tr\left(a\left(\id - 2 \boldsymbol{\tilde{\Sigma}}_t a \right)^{-1}{\boldsymbol{\dot{\tilde{\Sigma}}}_t} \right)= \Tr\left( a \left(\id - 2 \boldsymbol{\tilde{\Sigma}}_t a \right)^{-1}(\id -b\boldsymbol{K})^{-1}{\boldsymbol{\dot{{\Sigma}}}_t}(\id -b\boldsymbol{K}^*)^{-1}\right).
\end{align}
Finally, using \eqref{def:riccati_operator} and the identity $\Tr(\boldsymbol{F}\boldsymbol{G})=\Tr(\boldsymbol{G}\boldsymbol{F})$, we obtain
\begin{align}\label{eq:riccati_phi}
\dot {\phi}_t  =  \Tr(\boldsymbol{\Psi}_t{\boldsymbol{\dot{{\Sigma}}}_t}).
\end{align}
When $u,w$ are complex numbers, the definition of \eqref{eq:detphi} requires the use of several branches of the complex logarithm. For numerical implementation, to prevent complex discontinuities, one should either use \eqref{eq:detphi} with multiple branches or stick with the discretization of expression \eqref{eq:phii}. We refer to section~\ref{S:numeric} for the numerical implementation.
\end{remark}	
}

Finally, for  $K(t,s)=\bm 1_{s<t}$ and an input curve of the form 
\begin{align}\label{eq:paramg0}
g_0(t)= X_0 + \theta t, \quad t\geq 0,
\end{align}
for some $X_0,\theta \in \R$,
 one  recovers  from Theorem~\ref{T:volterra} the well-known closed form expressions of \citet{stein1991stock} and \citet{schobel1999stochastic}, and  that of \citet{heston1993closed} when $\theta= 0$.
 
 \begin{corollary}
 	Assume that $K(t,s)=\bm 1_{s<t}$ and that $g_0$  is of the form \eqref{eq:paramg0}, then, the expression \eqref{eq:charvolterra} reduces to 
 	\begin{align}\label{eq:charvolterrastandard}
 	\!\!\!\!\!\!\!\!\!\!\!	\E\left[ \exp\left(u \log \frac{S_T}{S_t} + w \int_t^T X_s^2 ds  \right) \! \!\Mid \mathcal \! \mathcal F_t\right] = {\exp\left( A(t)+ B(t) X_t +  C(t)X_t^2 \right)}
 	\end{align}
 	where $A,B,C$ solve the following system of (Backward) Riccati equations 
 	\begin{align*}
 	\dot A &= -\theta B - \frac 1 2 \nu^2 B^2 - \nu^2 C, \quad &&A(T)=0,   \\
 		\dot B &= -2 \theta C - (\kappa+\rho\nu u+2 {\nu^2} C)B, \quad &&B(T)=0, \\
 	\dot C &= -2\nu^2 C^2 - 2(\kappa+\rho\nu u)C - w- \frac 1 2 (u^2-u), \quad &&C(T)=0.
 	\end{align*}
 	In particular, $(A,B,C)$ can be computed in closed form as in \citet[Equations (43)-(44)-(45)]{lord2006rotation}.
 \end{corollary}

\begin{proof}[Sketch of proof]
	The characteristic function is given by \eqref{eq:charvolterra}. Assume that  $K(t,s)=\bm 1_{s<t}$ and $g_0$ is as in \eqref{eq:paramg0}. Then, 
	$$ X_s = X_t + (s-t)\theta + \int_t^s \kappa X_udu + \int_t^s \nu dW_u, \quad s\geq t, $$
	so that taking conditional expectation yields
	$$  g_t(s)=\bm 1_{t\leq s} \left(X_t + (s-t)\theta\right).$$
	It follows that 
	\begin{align}
	\langle g_t, \boldsymbol{\Psi}_{t} g_t \rangle_{L^2}=  \tilde A(t) +  B(t)  X_t+  C(t) X_t^2
	\end{align}
	with 
	$$\tilde A(t)=\theta^2 \langle   \bm 1_{t\leq \cdot}(\cdot-t), \boldsymbol{\Psi}_{t}  \bm 1_{t\leq \cdot}(\cdot-t) \rangle_{L^2},\;  B(t) = 2\theta \langle   \bm 1_{t\leq \cdot}(\cdot-t), \boldsymbol{\Psi}_{t}  \bm 1_{t\leq \cdot} \rangle_{L^2},  \;  C(t)=  \langle \bm 1_{t\leq \cdot} , \boldsymbol{\Psi}_{t} \bm 1_{t\leq \cdot} \rangle. $$
Combined with \eqref{eq:detphi} and \eqref{eq:riccati_phi} below, we obtain \eqref{eq:charvolterrastandard}  with  $A$ such that $A_T=0$ and 
	$$ \dot A(t)=  \dot {\tilde A}(t) + \Tr(\boldsymbol{\Psi}_t{\boldsymbol{\dot{{\Sigma}}}_t}) $$ 
	with  $\Tr$ the trace of an operator (see Appendix~\ref{A:trace} below) and 
	$$ \dot{{\Sigma}}_t(s,u)= -\nu^2 \bm 1_{t\leq s \wedge u}. $$
	Using the operator Riccati equation satisfied by $t \mapsto \boldsymbol \Psi_t$, see Lemma~\ref{L:Psi} below, and  straightforward computations as in \citet[Corollary 5.14]{abi2020markowitz} lead to the claimed system of Riccati equations for  $(A,B,C)$.
\end{proof}

\section{Numerical illustration}
In this section, we make use of the analytic expression for the characteristic function in \eqref{eq:introchar}  to price options.  We first present an approximation of the formula  \eqref{eq:introchar} using closed form expressions obtained from a natural discretization of the operators. Throughout this section, we consider the case $t=0$ and we fix   a Volterra kernel $K$, i.e.~$K(t,s)=0$ if $s\geq t$, as in Section~\ref{S:volterra}.

\subsection{A straightforward approximation by closed form expressions}\label{S:numeric} 
 The expression  \eqref{eq:introchar} lends itself to approximation by closed form solutions by a simple discretization of the operator $\boldsymbol{\Psi}_0$ given by \eqref{def:riccati_operator} \textit{à la} \cite{fredholm1903}.  Fix $n \in \mathbb N$ and let $t_i=iT/n$, $i=0,1,\ldots,n$ be a partition of $[0,T]$. Discretizing the  $\star$-product given in \eqref{eq:starproduct} yields the following approximation for $\boldsymbol{\Psi}_0$ by the $n\times n$ matrix:
\begin{align*}
\Psi_0^n = a\left(I_n - b (K^n)^\top \right)^{-1} \left(I_n - 2 \frac{aT}{n} \tilde \Sigma^n\right)^{-1} \left( I_n - b K^n \right)^{-1},
\end{align*}
where $I_n$ is the $n\times n$ identity matrix, $K^n$ is the lower triangular matrix with components 
\begin{align}\label{eq:Kn}
K^n_{ij}= \bm 1_{j\leq i-1}\int_{t_{j-1}}^{t_j} K(t_{i-1},s)ds  , \quad 1 \leq  i,j\leq n, 
\end{align} 
and 
$$  \tilde \Sigma^n =  \left(I_n - b K^n \right)^{-1}\Sigma^n \left(I_n - b (K^n)^{\top} \right)^{-1}   $$
with $\Sigma^n$ the $n \times n$ discretized covariance matrix, recall \eqref{eq:sigmakernel}, given by
\begin{align}\label{eq:Sigman}
 \Sigma^n_{ij}=\nu^2 \int_0^T K(t_{i-1},s)K(t_{j-1},s)ds, \quad 1\leq i,j\leq n.
\end{align}  
Defining the $n$-dimensional vector $g_n=(g_0(t_0),\ldots, g_0(t_{n-1}))^\top$, the discretization of the inner product $\langle \cdot, \cdot\rangle_{L^2}$ leads to the approximation 
\begin{align}\label{eq:approxchar}
\!\!\!\!\!\!\!\!	\E\left[ \exp\left(u \log {S_T} + w \int_0^T X_s^2 ds  \right) \right] \approx \frac{\exp\left(u\log S_0+ \frac{T}{n} g_n^\top {\Psi}_0^n g_n  \right)}{\det(\Phi^n_0)^{1/2}}
\end{align}
with $\Phi^n_0=\left( I_n - 2 a \frac{T}{n} \tilde \Sigma^n  \right)$.

{\begin{remark}
		Recalling Remark~\ref{rmk:det}, one needs to be careful with the numerical implementation of the square root of the determinant that appears in equation \eqref{eq:approxchar} to avoid complex discontinuities,  either by switching the sign of the determinant each time it crosses the axis of negative real numbers or by discretizing \eqref{eq:phii} which would require the computation of $\boldsymbol{\Psi}_t$ for several values of $t$ but has the advantage of being analytic on the whole domain. We refer to \cite{mayerhofer2019reforming} for more details for finite-dimensional Wishart distributions. 
\end{remark}}

\begin{remark}\label{R:nystrom}
Depending on the smoothness of the kernel, other quadrature rules  might be more efficient for the choice of the discretization of the operator and the approximation of the Fredholm determinant based on the so-called Nystr{\"o}m method, see for instance \citet{bornemann2009numerical,bornemann2010numerical,corlay2010nystr,kang2003nystrom}.
\end{remark}

\begin{remark}
	{For the case $u=0$ and $\kappa = 0$, the previous approximation formulas agree with the ones derived in \cite[Section 2.3]{AJ2019laplace} where  a numerical  illustration for the integrated squared fractional Brownian motion is  provided.}
\end{remark}

 \subsection{Option pricing in the fractional Stein--Stein model}\label{S:pricing}
 In this section, we illustrate the applicability of our results on the following  fractional Stein--Stein model based on the Riemann--Liouville fractional Brownian motion with the Volterra convolution kernel $K(t,s)=\bm{1}_{s< t}  (t-s)^{H-1/2}/{\Gamma(H+1/2)}$:
 \begin{align}
 d S_t &= S_t X_t  dB_t, \quad S_0>0, \label{eq:steinsteinS2}\\
 X_t &= g_0(t) +  \frac{\kappa}{\Gamma(H+1/2)}\int_0^t (t-s)^{H-1/2}   X_sds  + \frac{\nu}{\Gamma(H+1/2)}\int_0^t (t-s)^{H-1/2}   dW_s,  \label{eq:steinsteinX2} 
 \end{align}
 with $B= \rho W + \sqrt{1-\rho^2} W^{\perp}$, for $\rho \in [-1,1]$,  $\kappa, \nu \in \R$ and a Hurst index $H\in(0,1)$.  For illustration purposes we will consider that the input curve $g_0$, which can be used in general to fit at-the-money curves observed in the market, has the following parametric form\footnote{In conventional Markovian stochastic volatility models, the input curve $g_0$ is usually in the parametric form \eqref{g0param}.  	However, if one is interested in a practical implementation, then more general forms of  $g_0$ (non-parametric) would allow more flexibility (by making $\theta$ time dependent for instance). The advantage is that $g_0$   can be estimated from the market to match certain term structures today (e.g. term structure of forward variance, etc\ldots).  	For illustration purposes here, and since a comparison with the standard Stein--Stein model is given, we restrict to such parametric forms of $g_0$.}
 \begin{align}\label{g0param}
 g_0(t)&= X_0 +  \frac{1}{\Gamma(H+1/2)}\int_0^t (t-s)^{H-1/2} \theta ds =X_0 +  \theta  \frac{t^{H+1/2}}{\Gamma(H+1/2)(H+1/2)}.
 \end{align}

 \begin{remark}\label{R:fbm}
It would have also been possible to take  instead of the fractional Riemman--Liouville Brownian motion the true fractional Brownian motion by considering 
$$ X_t = g_0(t)+  \frac{\nu}{\Gamma(H+1/2)}\int_0^t (t-s)^{H-1/2}  \, {}_2F_1\left(H-1/2,1/2-H;H+1/2,1-\frac t s \right)  dW_s, $$
where   ${}_{2}F_{1}$ is the Gaussian hypergeometric function.
 \end{remark}
 
 Taking $H<1/2$  allows one to reproduce the stylized facts observed in the market as in Figure~\ref{fig:skew intro}. Indeed, the simulated sample paths of the  instantaneous variance process $X^2$ with $H=0.1$ in Figure~\ref{fig:simulated} has the  same regularity as the realized variance of the S\&P in Figure~\ref{fig:skew intro}-(a). In the case $H<1/2$, we  refer to the model as the  rough Stein--Stein model. 
 \begin{center}
 	\includegraphics[scale=0.6]{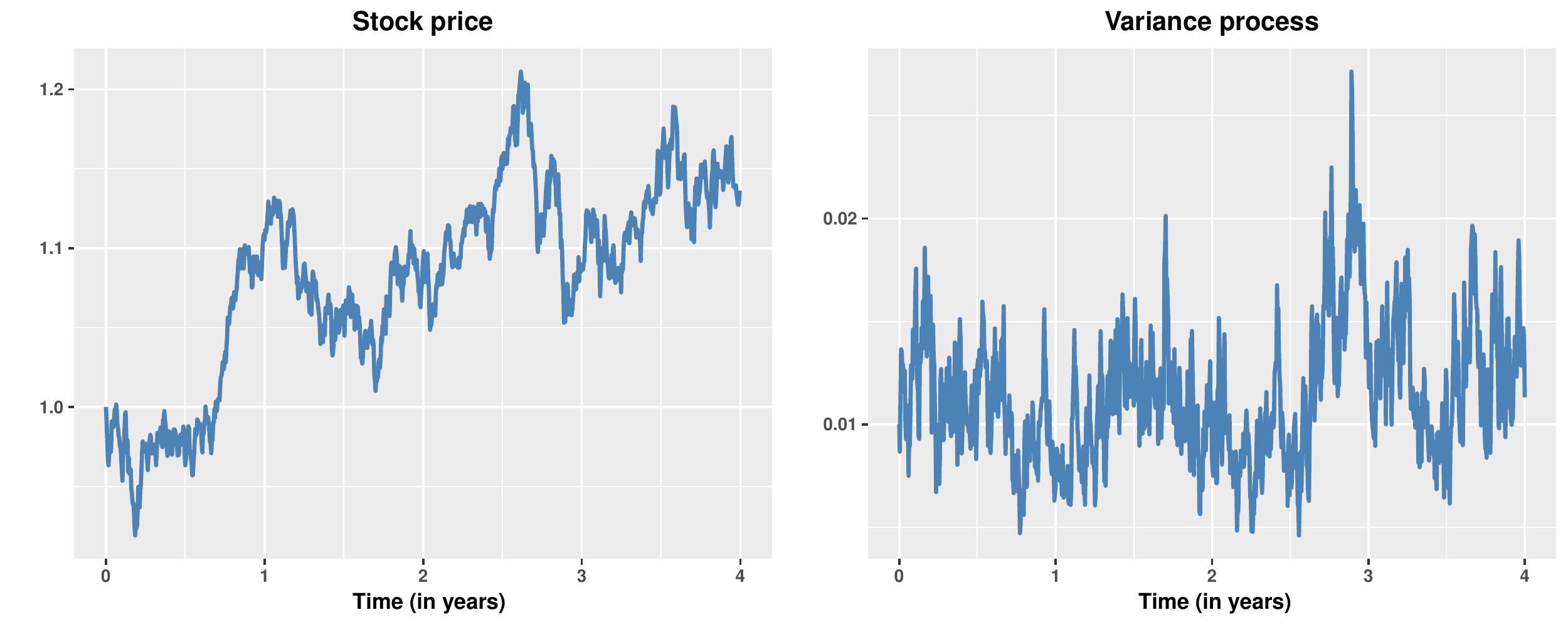}
 	\captionof{figure}{One simulated sample path of the stock price $S$ and the instantaneous variance process $X^2$ in the rough Stein--Stein model with parameters: $X_0=0,1$, $\kappa = 0$, $\theta=0.01$, $\nu=0.02$, $\rho=-0.7$ and $H=0.1$.}
 	\label{fig:simulated}
 \end{center}

We now move to pricing. The expression \eqref{eq:introchar} for the joint characteristic function  allows one to recover the joint density $p_T(x,y)$ of $\left(\log S_T,\int_0^T X_s^2 ds\right)$ by Fourier inversion:
\begin{align}
p_T(x,y)=\frac 1 {2\pi} \int_{\R^2} e^{-i(z_1x+z_2y)} \E\left[ \exp\left(iz_1\log S_T+iz_2\int_0^T X_s^2 ds\right) \right]dz_1 dz_2,
\end{align}
but also to price  derivatives on the stock price and the integrated variance by Fourier inversion techniques, see \citet{carr1999option,fang2009novel,lewis2001simple} among many others. In the sequel we will make use of the cosine method of \citet{fang2009novel}  to price European call options on the stock $S$ combined with our approximation formulae of Sections \ref{S:numeric}. We start by observing that the kernel $\Sigma_0$ is given in the following closed form 
\begin{align}
\Sigma_0(s,u)&=\frac{\nu^2}{\Gamma(H+1/2)^2}\int_0^{s\wedge u} (s-z)^{H-1/2}(u-z)^{H-1/2}dz\\
&=\frac{\nu^2}{\Gamma(\alpha)\Gamma(1+\alpha)}\frac{s^{\alpha}}{u^{1-\alpha}} \; {}_{2}F_{1}\left( 1, 1-\alpha; 1+\alpha ; \frac s u\right), \quad s\leq u, 
\end{align}
where $\alpha=H+1/2$ and  ${}_{2}F_{1}$ is the Gaussian hypergeometric function,
  see for instance \citet[page 71]{malyarenko2012invariant}.\footnote{Note that in the case of Remark~\ref{R:fbm}, the expression for the covariance function simplifies to $\Sigma_0(s,u)=\frac {\nu^2} 2 (s^{2H}+u^{2H}-|s-u|^{2H})$.}  {Fix $n \in \mathbb N$ and a  given partition $0=t_0<t_1 <\ldots <t_n=T$.  It follows that the $n\times n$ matrices \eqref{eq:Kn}-\eqref{eq:Sigman}  can be computed in closed form:
  \begin{align}
  K^n_{ij}&= \bm 1_{j\leq i-1} \frac{1}{\Gamma(1+\alpha)} \left[ (t_{i-1}-t_{j-1})^{\alpha} - (t_{i-1}-t_{j})^{\alpha} \right], \quad 1 \leq  i,j\leq n,  \\
  \Sigma_{ij}^n & = \frac{\nu^2}{\Gamma(\alpha)\Gamma(1+\alpha)}\frac{{t_{i-1}}^{\alpha}}{t_{j-1}^{1-\alpha}} \; {}_{2}F_{1}\left( 1, 1-\alpha; 1+\alpha ; \frac {t_{i-1}} {t_{j-1}}\right), \quad \Sigma_{ji}^n = \Sigma_{ij}^n,  \quad 1\leq i \leq j \leq n,
    \end{align}	
  with the convention that $0/0=0$. We note that $K^n$ is lower triangular with zeros on the diagonal and that the symmetric matrix $\Sigma^n$ has zeros on its first row and first column. The final ingredient to  compute \eqref{eq:approxchar} is the vector $g_n$ whose elements are given by:
 $$g_n^i = g_0(t_{i-1}) = X_0 +  \theta  \frac{{t_{i-1}}^{\alpha}}{\Gamma(1+\alpha)}, \quad 1\leq i \leq n.$$
  }

As a sanity check, we visualize on Figure~\ref{fig:convergence} the convergence of the approximation methods on the implied volatility for $H=0.2$ and $H=0.5$ with the uniform partition $t_i = iT/n$. The benchmark is computed for $H=0.5$ via the cosine method with   the closed form expressions for the characteristic function of the conventional Stein--Stein model, see \citet{lord2006rotation}; and for $H=0.2$ using Monte Carlo simulation. The smaller the maturity the faster the convergence.  Other discretization rules   might turn out to be more efficient and would require less points to achieve the same accuracy, which makes the implementation even faster, recall Remark~\ref{R:nystrom}. The main challenge for applying such methods is the singularity of the kernel at $s=t$ when $H<1/2$ and is left for future research.

\begin{center}
	\includegraphics[scale=0.65]{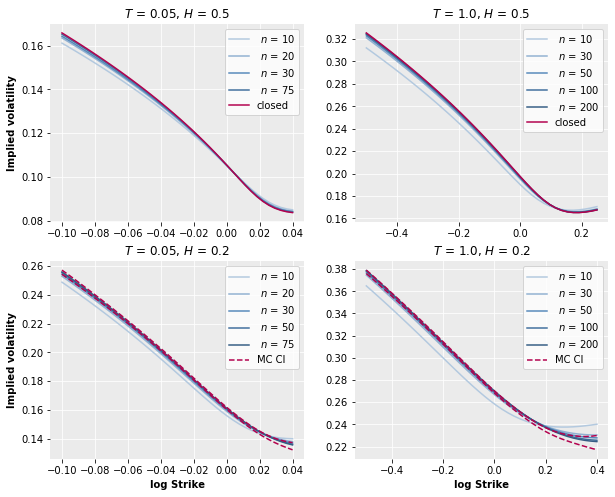}
	\rule{35em}{0.5pt}
	\captionof{figure}{{Convergence of the implied volatility slices for short ($T=0.05$ year) and long maturities ($T=1$ year)  of the operator discretization of  Section~\ref{S:numeric}  towards: (i) the explicit solution of the conventional Stein--Stein model ($H=0.5$ upper graphs); (ii) the $95\%$ Monte-Carlo confidence intervals ($H=0.2$ lower graphs). The parameters are $X_0=\theta = 0.1$, $\kappa=0$, $\nu=0.25$ and $\rho=-0.7$.}}
	\label{fig:convergence}
\end{center}

Going back to real market data, we calibrate the fractional Stein--Stein model to

\begin{enumerate}
	\item 
	the at-the-money skew of Figure \ref{fig:skew intro}-(b). Keeping the parameters
	$X_0=0.44$, $\theta=0.3$, $\kappa = 0$ fixed, the calibrated parameters are given by 
	\begin{align}\label{eq:calibrated skew rough}
	\hat \nu =0.5231458 , \quad \hat \rho= -0.9436174 \quad \mbox{and} \quad \hat H=  0.2234273.
	\end{align}
	This  power-law behaviour of  the at-the-money skew observed on the market is perfectly captured by the fractional Stein--Stein model as   illustrated on Figure \ref{fig:skew calibrated} with only three parameters. 
	\item
	the implied volatility surface of the S\&P accross several maturities for in Figure~\ref{fig:ivcalibrated}.
\end{enumerate}  
Both calibration lead to $\hat H< 0.5 $ indicating that the rough regime of the fractional Stein--Stein model is   coherent with the observations on the market.

\begin{center}
	\includegraphics[scale=0.5]{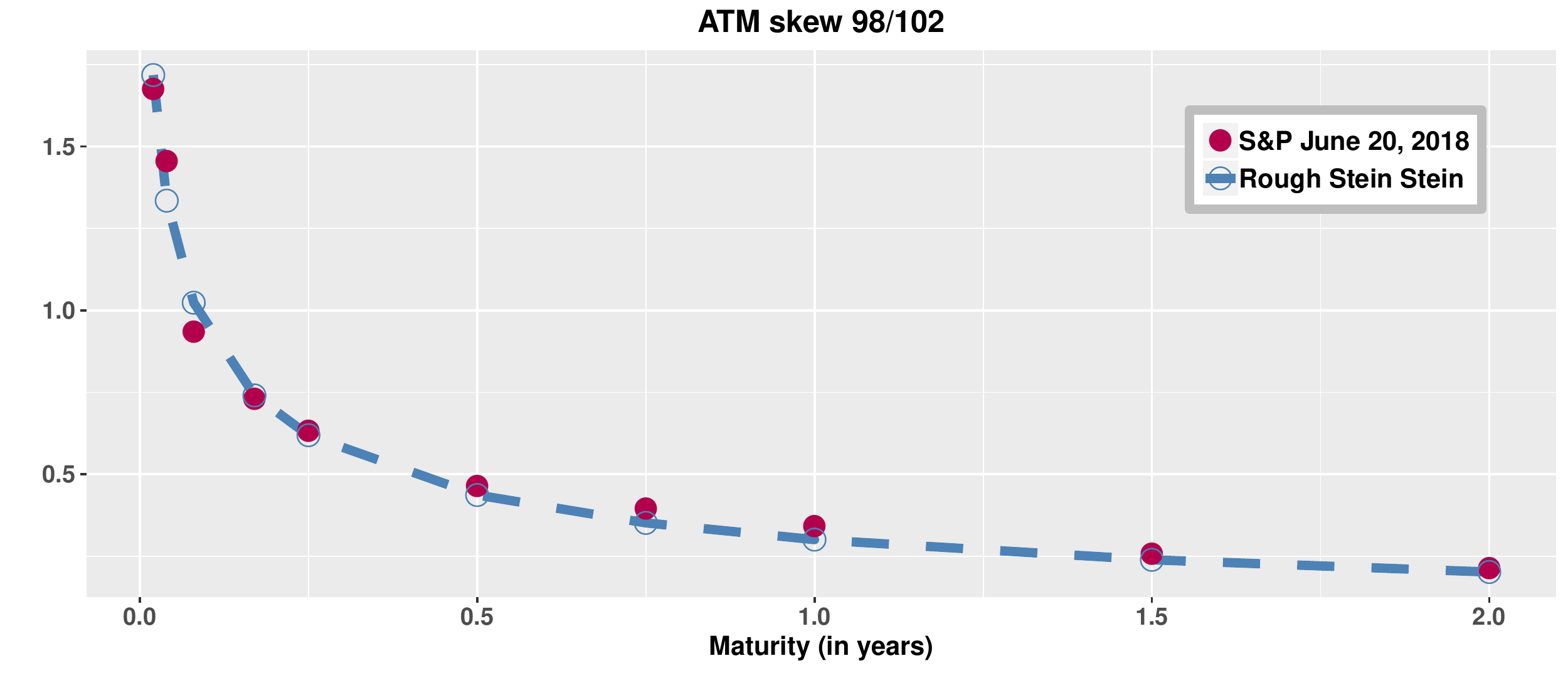}
	\rule{35em}{0.5pt}
	\captionof{figure}{{Term structure of the at-the-money skew for the S\&P index on June 20, 2018 (red dots) and for the rough Stein--Stein model with calibrated parameters \eqref{eq:calibrated skew rough} (blue circles with dashed line).} }
	\label{fig:skew calibrated}
\end{center}

\begin{center}
	\includegraphics[scale=0.5]{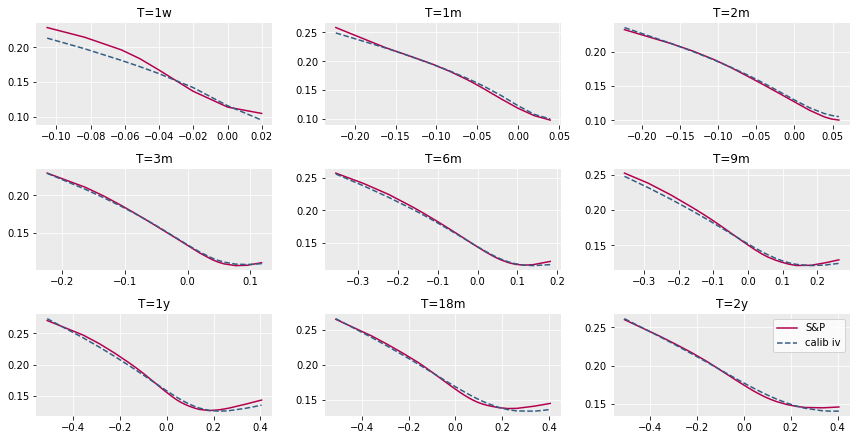}
	\rule{35em}{0.5pt}
	\captionof{figure}{{The implied volatility surface of the S\&P index (red) and the calibrated fractional Stein--Stein model (blue) with  parameters:  $ \hat X_0= 0.113,   \; \hat \theta=-0.044,   \;  \hat \kappa = -8.9 \e -5,\; 
			\hat \nu =  0.176, \; \hat \rho= -0.704, \; \mbox{and} \; \hat H=  0.279.$}}
	\label{fig:ivcalibrated}
\end{center}

\appendix


\section{Trace, determinants and resolvents}

\subsection{Trace and determinants}\label{A:trace}
{In this  section we recall classical results on operator theory in Hilbert spaces regarding mainly their trace and their determinant. For further details we refer to \citet{gohberg1978introduction, gohberg2012traces, simon1977notes, simon2005trace, smithies1958integral}, and also \citet[Section 2 and 3]{bornemann2009numerical}. 
	Let $\boldsymbol{A}$ be a linear compact operator acting on $L^2([0,T],\mathbb C)$. Then,  the operator $\boldsymbol{A}$ has a countable spectrum\footnote{We recall that the spectrum $\mbox{sp}(\boldsymbol{A})$ is defined as the set of points $\lambda \in \mathbb C$ for which there does not exist a bounded inverse operator $(\lambda \mbox{id} - \boldsymbol{A})^{-1}$.} denoted by $\mbox{sp}(\boldsymbol{A})=(\lambda_n(\boldsymbol{A}))_{n\leq N(\boldsymbol{A})}$, where $N(\boldsymbol{A})$ is either a finite integer or infinity. Whenever $\boldsymbol{A}$ is a linear  operator induced by a kernel   $A \in L^2([0,T]^2, \mathbb C)$,  $\boldsymbol{A}$ is a Hilbert--Schmidt operator on $L^2([0,T],\mathbb C)$ into itself and is in particular compact.

The trace and the determinant are two important  functionals  on the space of compact operators. Such quantities are defined for operators  of trace class. A compact operator $\boldsymbol{A}$ is said to be of trace class if the quantity 
	\begin{align}\label{eq:trace}
	\Tr \boldsymbol{A}  = \sum_{n \geq 1} \langle \boldsymbol{A} v_n , v_n  \rangle
	\end{align}
	is finite for a given orthonormal basis  $(v_n)_{n\geq 1}$. It can be shown that the quantity on the right hand side of \eqref{eq:trace} is independent of the choice of the orthonormal basis and will be called the trace of the operator $\boldsymbol{A}$. Furthermore, Lidskii's theorem \citep[Theorem 3.7]{simon2005trace} ensures that   
	\begin{align*}
	\Tr \boldsymbol{A}  = \sum_{n=1}^{N(\boldsymbol{A})} \lambda_n(\boldsymbol{A}).
	\end{align*}  

\begin{remark}
The product of two Hilbert-Schmidt operators $\boldsymbol{K}$ and $\boldsymbol{L}$ is of trace class. If in addition, both $\boldsymbol{K}$ and $\boldsymbol{L}$ are integral operators  on $L^2([0,T])$ induced by $K$ and $L$, then 
\begin{align}\label{eq:traceprod}
\Tr(\boldsymbol{K}\boldsymbol{L}) = \int_0^T (K\star L)(s,s)   ds,
\end{align} 
see \citet[Proposition 3]{brislawn1988kernels}.	
\end{remark}

Furthermore, the equivalence 
	\begin{align*}
\displaystyle \prod_{n\geq 1} (1 +   |\lambda_n|)  < \infty \Leftrightarrow   \sum_{n \geq 1} |\lambda_n| <\infty,
	\end{align*}
	allows one to define a determinant functional for a trace class operator $ \boldsymbol{A}$ by 
	\begin{align}
	\det (\id + z \boldsymbol{A}) =  \displaystyle \prod_{n=1}^{N(\boldsymbol{A})}(1 + z \lambda_n(\boldsymbol{A})),
	\end{align}
	for all $z\in \mathbb C$. If in addition $\boldsymbol{A}$ is an integral operator  induced by a continuous kernel $A$, then one can show that	
	\begin{align}\label{eq:detfredholm}
	\det(\id + z \boldsymbol{A})= \sum_{n \geq 0} \frac {z^n} {n!} \int_0^T \ldots \int_0^T \det \left[ (A(s_i,s_j))_{1\leq i,j\leq n} \right] ds_1\ldots ds_n.
	\end{align}		 
			The determinant \eqref{eq:detfredholm} is named after \citet{fredholm1903} who defined it for the first time for integral operators with continuous kernels. 
		
}

\vspace{1mm}

\vspace{1mm}

\subsection{Resolvents}\label{A:res}
For a kernel $K \in L^2([0,T]^2,\mathbb C)$, we define its resolvent $R_T \in L^2([0,T]^2,\mathbb C)$ by the unique solution to 
\begin{align}\label{eq:resolventeqkernel}
R_T = K + K \star R_T, \quad  \quad  K \star R_T =  R_T \star K.
\end{align} 
In terms of integral operators, this translates into 
\begin{align}\label{eq:resolventeqop0}
\boldsymbol{R}_T =  \boldsymbol{K} + \boldsymbol{K}\boldsymbol{R}_T, \quad \boldsymbol{K}\boldsymbol{R}_T=\boldsymbol{R}_T\boldsymbol{K}.
\end{align}
In particular, if $K$ admits a resolvent,  $({\id}-\boldsymbol{K})$ is invertible and
\begin{align}\label{eq:integralresop}
({\id}-\boldsymbol{K})^{-1}=\id+\boldsymbol{R}_T.
\end{align}


{\begin{lemma}\label{L:res}
	Any $K$ as in Definition~\ref{D:kernelvolterra} admits a resolvent kernel $R_T$ which is again a    Volterra kernel and satisfies~\eqref{D:kernelvolterra}.
\end{lemma}

\begin{proof}
It follows from \eqref{D:kernelvolterra}  that   $K$ is a   Volterra kernel of continuous and bounded type in the terminology of \citet[Definitions 9.5.1 and 9.5.2]{GLS:90}. But since, we are considering kernels on the compact set $[0,T]$, then every kernel of bounded and continuous type is of bounded and uniformly continuous type, see \citet[p.243, paragraph 1]{GLS:90}.  An application of \citet[Theorem 9.5.5-(ii)]{GLS:90}, yields that $K$ admits a resolvent kernel $R_T$ which  is again a Volterra kernel of bounded and continuous type. In particular, 
\begin{align*}
|R_T|_{L^1}  := \sup_{t\leq T} \int_0^T |R_T(t,s)| ds < \infty.  
\end{align*}
It remains to prove that $R_T$ inherits condition \eqref{D:kernelvolterra} from $K$  using the resolvent equation \eqref{eq:resolventeqkernel}. We first show that 
\begin{align}\label{eq:RboundedL2}
\int_0^T \int_0^T |R_T(t,s)|^2 dt ds<\infty. 
\end{align}
An application of Jensen's inequality on the normalized measure $|R_T(t,z)|dz/\int_0^T|R_T(t,z')|dz'$ yields
\begin{align}
 \int_0^T \int_0^T |(R_T\star K)(t,s)|^2 dt ds \leq |R_T|_{L^1}   \sup_{r \leq T}  \int_0^T |K(r,s)|^2 ds \int_0^T \int_0^T  |R_T(t,z)|  dt   dz <\infty. 
\end{align}
Combined with the resolvent equation \eqref{eq:resolventeqkernel} and the first part of \eqref{D:kernelvolterra}, we obtain  \eqref{eq:RboundedL2}. 
Using \eqref{eq:RboundedL2} and the Cauchy-Schwarz inequality we now get 
$$ \int_0^T |(K\star R_T )(t,u)|^2du \leq \sup_{t'\leq T} \int_0^T |K(t',z)|^2 dz \int_0^T \int_0^T |R_T(u,z)|^2 du dz< \infty, \quad t\leq T,$$
which combined with \eqref{D:kernelvolterra} and \eqref{eq:RboundedL2} gives 
\begin{align} 
 \sup_{t\leq T} \int_0^T |R_T(t,s)|^2 ds<\infty,
 \end{align}
 which  shows that $R_T$ satisfies the first condition in \eqref{D:kernelvolterra}.
 Finally, another application of the  Cauchy-Schwarz inequality, for all $t,h\geq 0$, shows that 
\begin{align*}
\int_0^T |(K\star R_T )(t+h,s) - (K\star R_T )(t,s) |^2 ds  \leq  \int_0^T \int_0^T |R_T(u',s)|^2 du' ds \int_0^T |K(t+h,u)-K(t,u)|^2 du  
\end{align*}
where the left hand side goes to $0$ as $h\to 0$ from the second part of \eqref{D:kernelvolterra}. 
Combined with the resolvent equation \eqref{eq:resolventeqkernel}, we can deduce that 
\begin{align*}
\lim_{h\to 0} \int_0^T |R_T(t+h,u)-R_T(t,u)|^2 du = 0,
\end{align*}
which yields the second condition in \eqref{D:kernelvolterra} for $R_T$.
\end{proof}
}

Using the resolvent we can provide the explicit solution to the system \eqref{eq:steinsteinS}--\eqref{eq:steinsteinX}.


{\begin{theorem}\label{T:existenceappendix}
	Fix $T>0$, $g_0\in L^2([0,T],\R)$ and a kernel $K$ as in Definition~\ref{D:kernelvolterra}. Then, there exists a {unique} progressively measurable strong solution $(X,S)$ to \eqref{eq:steinsteinS}--\eqref{eq:steinsteinX} on $[0,T]$ 
	given by 
	\begin{align}
	X_t &=g_0(t)+ \int_0^t R^{\kappa}_T(t,s)g_0(s)ds +\frac{1}{\kappa} \int_0^t R^{\kappa}_T(t,s)\nu dW_s, 	\label{eq:Yclosed} \\
	S_t &= S_0 \exp\left( -\frac{1}{2}\int_0^t X_s^2 ds + \int_0^t X_s dB_s \right),\label{eq:Sclosed}
	\end{align}
	where $R^{\kappa}_T$ is the resolvent kernel of $\kappa K$ with the convention that  $R^{\kappa}_T/\kappa = K$ when $\kappa = 0$. 
	In particular,  \eqref{eq:momentsX} holds. 
\end{theorem}

\begin{proof} If $\kappa = 0$, the existence is trivial. Fix $\kappa \neq 0$. 
	An application of Lemma~\ref{L:res} on the kernel $\kappa K$ yields the existence of a resolvent $R^\kappa_T$ satisfying \eqref{D:kernelvolterra}. We define $X$ as in \eqref{eq:Yclosed} and we write it in compact form:
\begin{align}
X&= (\id + \boldsymbol{R}^{\kappa}_T)(g_0)+  \frac 1 {\kappa} \boldsymbol{R}^{\kappa}_T (\nu dW),
\end{align}	
where we used the notation $\boldsymbol{R}^{\kappa}_T(\nu dW)(t)= \int_0^t {R}^{\kappa}_T (t,s)\nu dW_s:=N_t$.  {We first observe that $X$ admits a progressively measurable modification. Indeed, the stochastic integral $N$ is adapted as an Itô integral and it is 	mean-square continuous, i.e. $\E[|N_t -N_s|^2] \to 0$ as $s\to t$ by virtue of Itô's isometry and the fact that $R_T^{\kappa}$ satisfies \eqref{D:kernelvolterra} (see Lemma~\ref{L:res}). Therefore, $N$ admits a progressively measurable modification.} We now show that $X$ solves \eqref{eq:steinsteinX}. 
Using \eqref{eq:integralresop}, composing both sides by $(\id + \boldsymbol{R}^{\kappa}_T)^{-1} = (\id - \kappa \boldsymbol{K})$ and invoking stochastic Fubini's theorem yield 
\begin{align*}
(\id -\kappa\boldsymbol{K})(X) &= g_0  +  (\id -\kappa\boldsymbol{K}) \frac 1 {\kappa} \boldsymbol{R}^{\kappa}_T (\nu dW) \\
&= g_0 +  \boldsymbol{K}  (\nu dW),
\end{align*} 
where we used the resolvent equation \eqref{eq:resolventeqop0} for the last equality. 
This shows that 
$$ X_t = g_0(t) + \kappa (\boldsymbol{K})(X)(t) + \boldsymbol{K}(\nu dW)(t)= g_0(t) + \kappa\int_0^t K(t,s) X_s ds + \int_0^tK(t,s)\nu dW_s, $$ 
 yielding that $X$ is a  strong solution of \eqref{eq:steinsteinX}. 
 Furthermore, \eqref{eq:momentsX} follows from the fact that $\sup_{s\leq T} \int_0^T |R^\kappa_T(s,u)|^2 du<\infty$  combined with the Burkholder-Davis-Gundy inequality.
 One can therefore define $S$ as in \eqref{eq:Sclosed} and it is immediate that $S$ solves \eqref{eq:steinsteinS} by an application of It\^o's formula. The uniqueness statement follows by reiterating the same argument backwards: by showing that any solution $X$ to \eqref{eq:steinsteinX} is of the form \eqref{eq:Yclosed} using  the resolvent equation.
\end{proof}
}

We now justify in the three following lemmas that  the quantities  $(\id -b\boldsymbol{K})$ and $\left( \id - 2 a\boldsymbol{\tilde \Sigma}_{t}  \right)$ appearing in the definition of $t\mapsto\boldsymbol{\Psi}_t$ in \eqref{def:riccati_operator} are invertible  so that $\boldsymbol{\Psi}_t$ is well-defined for any kernel $K$ as in Definition~\ref{D:kernelvolterra}.

\begin{lemma}\label{L:staroperation}
	Let $K$ satisfy \eqref{assumption:K_stein} and $L \in L^2([0,T]^2,\R)$. Then, $ K\star L$ satisfies  \eqref{assumption:K_stein}. Furthemore,  if $L$ satisfies  \eqref{assumption:K_stein}, then, $(s,u)\mapsto (K \star L^*)(s,u)$ is continuous.
\end{lemma}
\begin{proof}
	An application of the Cauchy-Schwarz inequality yields the first part. The second part follows along the same lines as in the proof of \citet[Lemma 3.2]{AJ2019laplace}.
\end{proof}

\begin{lemma}\label{L:Kt}
		Fix $b\in \mathbb C$ and	 a kernel $K$ as in Definition~\ref{D:kernelvolterra}.  Then, $(\id -b\boldsymbol{K} )$ is invertible. Furthermore, for all $t\leq T$, $\tilde{ \boldsymbol{\Sigma}}_t$ given by \eqref{def:C_tilde} {is an integral operator of trace class with continuous kernel} and can be re-written in the form
		\begin{align}\label{eq:sigmaKt}
		\tilde{\boldsymbol{\Sigma}}_t =  (\id- b\boldsymbol{K}_t)^{-1} \boldsymbol{\Sigma}_t (\id-b \boldsymbol{K}_t^*)^{-1}
		\end{align}
		where $\boldsymbol{K}_t$ is the integral operator induced by the kernel $K_t(s,u)=K(s,u)\bm1_{u\geq t}$, for $s,u\leq T$. 
\end{lemma}

\begin{proof}
		Lemma~\ref{L:res} yields the existence of the resolvent $R^b_T$ of  $bK$  which is again a Volterra kernel of continuous and bounded type.  Whence,  \eqref{eq:integralresop} yields that $(\id -b\boldsymbol{K} )$ is invertible with an inverse given by $(\id + \boldsymbol{R}^b_T)$. 
		To prove \eqref{eq:sigmaKt}, we fix $t\leq T$ and we observe that since $\Sigma_t(s,u)=0$ whenever $s\wedge u \leq t$, we have
		$$ (R^b_T \star \Sigma_t)(s,u) = {\int_t^T R^b_T(s,z)\Sigma_t(z,u)dz =  (R^b_{t,T}\star \Sigma_t) (s,u),} $$
		where we defined the kernel $R^b_{t,T}(s,u)=R^b_{T}(s,u)\bm{1}_{u\geq t}$. Similarly, $\Sigma_t \star (R^b_T)^* = \Sigma_t \star (R^b_{t,T})^*$. Using the resolvent equation \eqref{eq:resolventeqkernel} of $R^b_T$, it readily follows that   $R^b_{t,T}$ is the resolvent of $bK_t$ so that $(\id -b \boldsymbol{K_t})^{-1}=(\id + \boldsymbol{R}^b_{t,T})$. Combining all of the above leads to 
		\begin{align}
		\tilde{\boldsymbol{\Sigma}}_t &=  (\id- b\boldsymbol{K})^{-1} \boldsymbol{\Sigma}_t (\id-b \boldsymbol{K}^*)^{-1}  \\
		&= (\id + \boldsymbol{R}^b_T) \boldsymbol{\Sigma}_t (\id+ \boldsymbol{R}^b_T)^{*}\\
		&= \boldsymbol{\Sigma}_t +   \boldsymbol{R}^b_T \boldsymbol{\Sigma}_t +   \boldsymbol{\Sigma}_t (\boldsymbol{R}^b_T)^* +  \boldsymbol{R}^b_T \boldsymbol{\Sigma}_t (\boldsymbol{R}^b_T)^*\\
		&=  \boldsymbol{\Sigma}_t +   \boldsymbol{R}^b_{t,T} \boldsymbol{\Sigma}_t +   \boldsymbol{\Sigma}_t (\boldsymbol{R}^b_{t,T})^* +  \boldsymbol{R}^b_{t,T} \boldsymbol{\Sigma}_t (\boldsymbol{R}^b_{t,T})^* \label{eq:decompsig}\\
		&=  (\id + \boldsymbol{R}^b_{t,T}) \boldsymbol{\Sigma}_t (\id+ \boldsymbol{R}^b_{t,T})^{*}\\
		&=  (\id- b\boldsymbol{K}_t)^{-1} \boldsymbol{\Sigma}_t (\id + b\boldsymbol{K}^*_t)^{-1},
		\end{align}
		which proves \eqref{eq:sigmaKt}. {Furthermore, it can be readily deduced from \eqref{eq:decompsig} that $\tilde{\boldsymbol{\Sigma}}_t$ is an integral operator of trace class with continuous kernel: the trace class property follows from  the fact that the product of two Hilbert-Schmidt operators is of trace class; the continuity of the kernel follows from the fact that both $K$ and $R^b_T$ satisfy \eqref{assumption:K_stein}, recall Lemma~\ref{L:res}.}
\end{proof}

\begin{lemma}\label{eq:Psiwelldefined}
	Fix $a,b \in \mathbb C$ such that  $\Re(a) \leq -\frac{\Im(b)^2}{2\nu^2}$.  Let $t\leq T$ and	  $K$ be a kernel as in Definition~\ref{D:kernelvolterra}.  Then, 
		  $(\id -2\boldsymbol{\tilde{\Sigma}}_ta)$ is invertible and $\boldsymbol{\Psi}_t$ given by \eqref{def:riccati_operator}  is well-defined. Furthermore, if $\Im(a)=\Im(b)=0$ then,  $\boldsymbol{\Psi}_t$ is a symmetric nonpositive operator  in the sense of Definition~\ref{D:nonnegative}.
\end{lemma}

\begin{proof}
$\bullet$ 
 Using Lemma~\ref{L:Kt}, we write 
$$(\id -2a\boldsymbol{\tilde{\Sigma}}_t) = (\id -b\boldsymbol{K}_t )^{-1} \boldsymbol{A}_t (\id -b\boldsymbol{K}^*_t)^{-1}  $$
with 
\begin{align}
\boldsymbol{A}_t&=  \left( \id -b\boldsymbol{K}_t \right)  \left( \id -b\boldsymbol{K}^*_t \right)  -2a\boldsymbol{\Sigma}_t\\
&=\id -b\boldsymbol{K}_t  -b\boldsymbol{K}_t^* + b^2 \boldsymbol{K}_t\boldsymbol{K}_t^*  -2a\boldsymbol{\Sigma}_t.
\end{align}
It suffices to prove that $\boldsymbol{A}_t$ is invertible, that is $0 \notin \mbox{sp}(\boldsymbol{A}_t)$. Taking real parts and observing that $\Sigma_t = \nu^2 \boldsymbol{K}_t \boldsymbol{K_t}^*$ yields 
\begin{align}
\Re(\boldsymbol{A}_t)
&=\id -\Re(b)\boldsymbol{K}_t  -\Re(b)\boldsymbol{K}_t^* + \Re(b)^2 \boldsymbol{K}_t\boldsymbol{K}_t^*  -\Im(b)^2 \boldsymbol{K}_t\boldsymbol{K}_t^* -2\Re(a)\boldsymbol{\Sigma}_t \\
&=   \left( \id -\Re(b)\boldsymbol{K}_t \right)  \left( \id -\Re(b)\boldsymbol{K}_t \right)^*  - \left(2 \Re(a) + \frac{\Im(b)^2}{\nu^2}\right)\boldsymbol{\Sigma}_t\\
&= \bold{I}+\bold{II}
\end{align}
The operator $ \bold{I}$ is symmetric nonnegative  and invertible so that $\mbox{sp}(\bold{I})\subset (0,\infty)$. Furthermore, since $\left(2 \Re(a) + \frac{\Im(b)^2}{\nu^2}\right) \leq 0$ by assumption and $\bold{\Sigma}_t$ is symmetric nonnegative we have $\mbox{sp}(\bold{II})\in [0,\infty) $. It follows that  $\mbox{sp}(\Re(\boldsymbol{A}_t))\in (0,\infty) $, showing that $0 \notin \mbox{sp}(\boldsymbol{A}_t)$ and that $\boldsymbol{A}_t$ is invertible. Whence,  $(\id -2a\boldsymbol{\tilde{\Sigma}}_t)$ is invertible. Combined with Lemma~\ref{L:Kt}, we obtain that $\boldsymbol{\Psi}_t$ is well-defined.  \\  
$\bullet$ Assume that $\Im(a)=\Im(b)=0$. $\boldsymbol{\tilde{\Sigma}}_t$ defined as in \eqref{def:C_tilde} is clearly a  symmetric nonnegative operator with a continuous kernel on $[0,T]^2$, recall Lemma~\ref{L:staroperation},  an application of Mercer's theorem \citep[Theorem 1, p.208]{shorack2009empirical} yields the existence of an orthonormal basis $(e_{n})_{n\geq 1}$ of $L^2([0,T],{\R})$ and nonnegative eigenvalues $(\lambda_{n})_{n\geq 1}$ such that 
	$$ \boldsymbol{\tilde{\Sigma}}_t = \sum_{n\geq 1} \lambda_{n} \langle e_{n}, \boldsymbol{\cdot}\, \rangle_{L^2} e_{n}.$$
	Whence,
	$$ \id -2a\boldsymbol{\tilde{\Sigma}}_t = \sum_{n\geq 1} (1-2a\lambda_{n}) \langle e_{n}, \boldsymbol{\cdot}\, \rangle_{L^2} e_{n}.$$
	Since $a\leq 0$, $(1-2a\lambda_{n})\geq 1 >0$, for each $n\geq 1$, so that the inverse of $(\id -2a\boldsymbol{\tilde{\Sigma}}_t)$ is a symmetric nonnegative operator  given by 
	$$ \left(\id -2a\boldsymbol{\tilde{\Sigma}}_t\right)^{-1} = \sum_{n\geq 1} \frac{1}{1-2a\lambda_{n}}\langle e_{n}, \boldsymbol{\cdot}\, \rangle_{L^2} e_{n}.$$
	Finally, $\boldsymbol{\Psi}_t$ is clearly symmetric and for any  $f \in L^2([0,T],\R)$
	$$ \langle f , \boldsymbol{\Psi}_tf \rangle_{L^2}= a \langle \tilde f ,  \left(\id -2a\boldsymbol{\tilde{\Sigma}}_t\right)^{-1} \tilde f \rangle_{L^2} \leq 0,$$
	with $\tilde f= (\id -b\boldsymbol{K})^{-1} f$.  This shows that $\boldsymbol{\Psi}_t$ is nonpositive. 
\end{proof}

\section{Proof of Theorem~\ref{T:volterra}}\label{A:proofvolterra}
This section is dedicated to the proof of Theorem~\ref{T:volterra}.	We fix $T>0$,   a Volterra kernel $K$ as in Definition~\ref{D:kernelvolterra} satisfying \eqref{eq:assumptionkerneldiff1} and $u,w\in \mathbb C$, such that $0\leq \Re(u)\leq 1$ and $\Re(w)\leq0$. It follows that $a,b$ defined by \eqref{eq:ab} satisfy   
$$ \Re(a)+\frac{\Im(b)^2}{2\nu^2} = \Re(w) + \frac 1 2 (\Re(u)^2-\Re(u))  +\frac  1 2 (\rho^2-1)\Im(u)^2 \leq 0, $$
so that  an application of Lemma~\ref{eq:Psiwelldefined}
	yields that $\boldsymbol{\Psi}_t$ is well-defined.

We now collect from \citet[Lemma 5.8]{abi2020markowitz} further properties of    $t\mapsto\boldsymbol{\Psi}_t$. In particular, its link with an operator Riccati equation. We recall that $t\mapsto \boldsymbol{\Psi}_t$ is said to be strongly differentiable at time $t\geq  0$,  if there exists a bounded linear operator $\dot{\boldsymbol{\Psi}}_t$  from $ L^2\left([0,T],{\C}\right)$  into  itself  such that
	\begin{align}
	\lim_{h\to 0} \frac{1}{h} \| \boldsymbol{\Psi}_{t+h}-\boldsymbol{\Psi}_{t} -h \dot{\boldsymbol{\Psi}}_{t} \|_{\rm{op}}=0, \quad \text{where }  \|\boldsymbol{G}\|_{\rm {op}}= \sup_{f \in L^2([0,T],\C)} \frac{\|\bold G f\|_{L^2}}{\|f\|_{L^2}}.
	\end{align}

	\begin{lemma}\label{L:Psi}
		Fix a kernel $K$ as in Definition~\ref{D:kernelvolterra} satisfying \eqref{eq:assumptionkerneldiff1}. Then, for each $t\leq T$, $\boldsymbol{\Psi}_t$ given by \eqref{def:riccati_operator}  is a bounded  linear operator from $L^2\left([0,T],\R\right)$ into itself. Furthermore, 
		\begin{enumerate}
			\item \label{L:Psi1}
			$\bar{\boldsymbol{\Psi}}_t:=(-a \id + \boldsymbol{\Psi}_t)$ is an integral  operator induced by a symmetric kernel  $\bar \psi_t(s,u)$ such that 
			\begin{align}
			\label{eq:bound_psi_leb}
			\sup_{t\leq T} \int_{[0,T]^2}|\bar \psi_t(s,u)|^2 ds du<\infty.
			\end{align}
			\item \label{L:Psi2}
			For any $f \in L^2\left([0,T],\R\right)$,  
			\begin{align}
			\label{eq:Psi_on_boudary}
			(\boldsymbol{\Psi}_t f 1_t)(t) =& (a \id + b\boldsymbol{{K}}^*\boldsymbol{\Psi}_t )(f 1_t)(t),
			\end{align}
			where $1_t:s\mapsto \bold 1_{t\leq s}$.
			\item \label{L:Psi3}
			$t\mapsto \boldsymbol{\Psi}_t$ is strongly differentiable and  satisfies the operator Riccati equation
			\begin{align}
			\label{eq:riccati_psiBold}
			\dot{\boldsymbol{\Psi}}_t &= 2\boldsymbol{\Psi}_t \dot{\boldsymbol{{\Sigma}}}_t  \boldsymbol{\Psi}_t, \qquad t \in [0,T] \\
			{\boldsymbol{\Psi}_T}&=a\left(\id - b\boldsymbol{{K}}^*\right)^{-1} \left(\id - b\boldsymbol{{K}}\right)^{-1}
			\end{align}  
			where $\dot{\boldsymbol{{\Sigma}}}_t$ is the strong derivative of $t\mapsto\bold{\Sigma}_t$ induced by the  kernel
			\begin{align}\label{eq:diffkernelsigma}
			\dot{\Sigma}_t(s,u)=-\nu^2 K(s,t) K(u,t) , \quad a.e.
			\end{align}
		\end{enumerate}
	\end{lemma}
	
	\begin{proof}
		The proof follows from a straighforward adaptation of the proof of \citet[Lemma 5.6]{abi2020markowitz}. 
	\end{proof}

Using the previous lemma and observing  that the adjusted conditional mean given in \eqref{eq:condmean} has the following dynamics 
\begin{align}\label{eq:processg_quadratic}
g_t(s)= \bm 1_{t\leq s} \left(g_0(s) + \int_0^t K(s,u)\kappa X_u du + \int_0^t K(s,u)\nu dW_u   \right)
\end{align}
we derive in the next lemma the dynamics of $t\mapsto \langle g_t, \boldsymbol{\Psi}_t g_t \rangle_{L^2}$.

	\begin{lemma}
		The dynamics of $t\mapsto \langle g_t, \boldsymbol{\Psi}_t g_t \rangle_{L^2}$ are given by
		\begin{align}
		d\langle g_t, \boldsymbol{\Psi}_t g_t \rangle_{L^2}&=  \Big(\langle g_t, \dot{\boldsymbol{\Psi}}_t g_t \rangle_{L^2} -aX_t^2 - 2 u\rho  \nu X_t \left( \boldsymbol{K}^* \boldsymbol{\Psi}_t\right)(g_t)(t)   - \Tr\big( \boldsymbol{\Psi}_t \dot{\boldsymbol{\Sigma}}_t\big) \Big) dt \\
		&\quad \quad   +  2\nu \left(\left(\boldsymbol{K}^* \boldsymbol{\Psi}_t\right)g_t\right)(t)  dW_t,  \quad dt\times \Q-a.e \label{eq:dynamicsL2inner}
		\end{align}
	\end{lemma}

		\begin{proof}
We first set 
\begin{align}\label{eq:barg}
\bar g_t(s) = g_0(s) + \int_0^t K(s,u)\kappa X_u du + \int_0^t K(s,u)\nu dW_u,
\end{align}   
so that using Lemma~\ref{L:Psi}-\ref{L:Psi1}, we can write
		\begin{align}\label{eq:quadfun}
		\langle g_t, \boldsymbol{\Psi}_t g_t \rangle_{L^2} =&  \int_t^T \left(a \bar g_t(s)^2 ds +     \bar g_t(s)  (\bar{\boldsymbol{\Psi}}_t g_t)(s)  \right)ds. 
		\end{align}	
The Leibniz rule yields 
		\begin{align}
		d\langle g_t, \boldsymbol{\Psi}_t g_t \rangle_{L^2} &=  \left(-a \bar g_t(t)^2  -\bar g_t(t)  (\bar{\boldsymbol{\Psi}}_t g_t)(t) \right) dt \nonumber \\
		&\quad   + \int_t^T  d \left(a\bar g_t(s)^2  + \bar g_t(s)  (\bar{\boldsymbol{\Psi}}_t g_t)(s)\right) ds,  
		\quad dt\times \Q \, \, a.e. \label{eq:leibniz}
		\end{align}
		$\bullet$ We first compute   the dynamics of $t\mapsto a \bar g_t(s)^2 ds +     \bar g_t(s)  (\bar{\boldsymbol{\Psi}}_t g_t)(s) $. We fix $s\in [0,T]$. It follows from \eqref{eq:barg}, that 
\begin{align}
d\bar g_t(s) &= K(s,t)\kappa X_t dt + K(s,t)\nu dW_t, \quad dt\times \Q-a.e.
\end{align}
An application of Itô's lemma on the square yields 
\begin{align}
d\bar g_t(s)^2 &= \left(\nu^2 K(s,t)^2 +  2 \bar g_t(s) K(s,t)\kappa X_t \right) dt + 2 \bar g_t(s) K(s,t)\nu dW_t, \quad dt\times \Q-a.e.
\end{align}
Furthermore, we write 
$$ (\bar{\boldsymbol{\Psi}}_t g_t)(s) = \int_t^T \bar \psi_t (s,u) \bar g_t(u)du$$
so that an application of the Leibniz rule combined with 
the fact that $\bar g_t(t)=X_t$ for almost every $(t,\omega)$ 
and  Lemma~\ref{L:Psi}-\ref{L:Psi3} yields that  
$t\mapsto (\bar{\boldsymbol{\Psi}}_t g_t)(s)$ is a semimartingale on $[0,s)$ with the following dynamics 
\begin{align*}
d(\bar{\boldsymbol{\Psi}}_tg_t)(s) &= \left( - \bar \psi_t(s,t) X_t  + \int_t^T \dot{\bar{\psi}}_t(s,u)\bar g_t(u)du + \int_t^T \bar{\psi}_t(s,u) K(u,t) \kappa X_u du  \right) dt \\
&\quad   + \int_t^T \bar{\psi}_t(s,u) K(u,t) \nu du  dW_t \\ 
&=  \left(-X_t \bar \psi_t(s,t) + (\dot{\boldsymbol{\Psi}}_tg_t)(s) +X_t(\bar{\boldsymbol{\Psi}}_t K(\cdot,t)\kappa)(s)\right) dt  \\
&\quad  +   (\bar{\boldsymbol{\Psi}}_t K(\cdot,t)\nu )(s)dW_t , \quad dt\times \Q-a.e.   
\end{align*}
where we used that $\dot{\bar {\boldsymbol{\Psi}}}_t =  \dot{\boldsymbol{\Psi}}_t$ and that $K(u,t)=0$ for all $u\leq t$. 
Moreover, the quadratic covariation between $t\mapsto \bar g_t(s)$  and $t\mapsto (\bar{\boldsymbol{\Psi}}_t g_t)(s)$ is given by
\begin{align*}
d\left[\bar g_{\cdot}(s), (\bar{\boldsymbol{\Psi}}_{\cdot} g_{\cdot})(s) \right]_t&= \; 
\nu^2 \int_0^T  \bar \psi_{t}(s,u) K({u},t) K({s},t)  du dt \\
&= \;   -   \int_0^T  \bar \psi_{t}(s,u)  \dot{\Sigma}_t(u,s) du dt \\
&= \;   - \big(\bar{\boldsymbol{\Psi}}_t  \dot{\Sigma}_t (\cdot, s)\big)(s)dt. 
\end{align*}
Whence, combining the previous three identities, we get the dynamics of $U_t(s):=  a\bar g_t(s)^2  + \left(\bar g_t(s)  (\bar{\boldsymbol{\Psi}}_t g_t)(s)\right) $:
\begin{align}
dU_t(s) &=  \; a d\bar g_t(s)^2  + d\bar g_t(s)   (\bar{\boldsymbol{\Psi}}_t g_t)(s)+ \bar g_t(s)   d(\bar{\boldsymbol{\Psi}}_t g_t)(s) + d\left[ \bar g_{\cdot}(s), (\bar{\boldsymbol{\Psi}}_{\cdot} g_{\cdot})(s) \right]_t  \\
&=\;  a\nu^2 K(s,t)^2 dt +  2 a \bar g_t(s) K(s,t)\kappa X_t dt \\
&\quad +  \; X_t  \kappa   K(s,t)    (\bar{\boldsymbol{\Psi}}_t g_t)(s)  dt   +   \bar g_t(s)   (\dot{\boldsymbol{\Psi}}_tg_t)(s)dt \\
& \quad - \; \bar g_t(s)  \bar{\psi}_t(s,t)X_tdt + \bar g_t(s) X_t  (\bar{\boldsymbol{\Psi}}_t K(\cdot,t) \kappa )(s) dt   \\
&\quad -  \;  \big(\bar{\boldsymbol{\Psi}}_t \dot \Sigma_t (\cdot, s)\big)(s)dt  \\
&\quad  + \;  \left(2 a\bar g_t(s) K(s,t)\nu +  \nu  K(s,t)    (\bar{\boldsymbol{\Psi}}_t g_t)(s)  +   \bar g_t(s)   (\bar{\boldsymbol{\Psi}}_t K(\cdot,t)\nu)(s)\right)dW_t  \\
&= \;  \Big [ \bold{I}(s)+\bold{II}(s)+ \bold{III}(s)+ \bold{IV}(s)+\bold{V}(s)+ \bold{VI}(s) + \bold{VII}(s)  \Big] dt \\
& \quad + \left(\bold{VIII}(s) + \bold{IX}(s) +  \bold{X}(s)\right)dW_t , \quad dt\times \Q-a.e.   \label{eq:dynamicsofsum}
\end{align}
$\bullet$ We now integrate  in $s$ to obtain the right hand side in \eqref{eq:leibniz}. We let $\mathcal N = \{(t,\omega) :  \exists s\in [0,T] \mbox{ such that } \eqref{eq:dynamicsofsum} \mbox{ does not hold} \}$. Then, $\mathcal N$ is a null set and we fix $(t, \omega) \in [0,T]\times \Omega \backslash \mathcal N$. In the sequel, all the equalities are written for this particular $\omega$.  First, using that $\dot{\Sigma}_t(s,s) = - \nu^2K(s,t)^2$  and recalling that 
\begin{align}\label{eq:linkbar}
\boldsymbol{\Psi} = a \id + \bar{\boldsymbol{\Psi}}
\end{align} 
we obtain that\footnote{The operator $\boldsymbol{\Psi}_t \dot{\boldsymbol{\Sigma}}_t = a \dot{\boldsymbol{\Sigma}}_t + \bar{ \boldsymbol{\Psi}}_t \dot{\boldsymbol{\Sigma}}_t $ is of trace class: (i)  $\dot{\boldsymbol{\Sigma}}_t$ is of trace class since it can be written as product of two Hilbert-Schmidt integral operators $\dot{\boldsymbol{\Sigma}}_t = \tilde{\boldsymbol{K}_t} \tilde{\boldsymbol{K}_t}^* $ with $\tilde{K}_t(s,z) = K(s,t)/\sqrt{T}$, so that  \eqref{eq:traceprod} yields $\Tr(\dot{\boldsymbol{\Sigma}}_t)=\int_0^T \dot{{\Sigma}}_t(s,s)ds$; (ii) $ \bar{ \boldsymbol{\Psi}}_t \dot{\boldsymbol{\Sigma}}_t$ is of  trace class as product of two Hilbert-Schmidt integral operators so that  \eqref{eq:traceprod} yields $\Tr( \bar{ \boldsymbol{\Psi}}_t  \dot{\boldsymbol{\Sigma}}_t)=\int_0^T \int_0^T \bar \psi_t(s,z) \dot{{\Sigma}}_t(z,s)dzds$.} 
$$\int_t^T\left(\bold{I}(s) + 	\bold{VII}(s)\right) ds\;  =  \; - \mbox{Tr}\Big( \boldsymbol{\Psi}_t \dot{\boldsymbol{\Sigma}}_t \Big).$$
Combining \eqref{eq:linkbar} with  Lemma~\ref{L:Psi}--\ref{L:Psi2} and the fact that  $\bar{\boldsymbol{\Psi}}^*=\bar{\boldsymbol{\Psi}}$   we obtain that  
\begin{align}
\int_t^T\big[  \bold{II}(s) +\bold{III}(s)  +\bold{VI}(s) \big] ds &= 2 \kappa X_t  \int_0^T  K(s,t) (\boldsymbol{\Psi}_t g_t)(s)  ds =  2\kappa X_t (\boldsymbol{K}^*\boldsymbol{\Psi}_t g_t) (t).
\end{align}
On the other hand, we have 
\begin{align}
\int_0^T \bold{IV}(s)ds &= \;  \langle g_t, \dot{\boldsymbol{\Psi}}_t g_t\rangle_{L^2}, \quad  \int_0^T  \bold{V}(s) ds = -X_t (\bar{\boldsymbol{\Psi}}_t g_t)(t), \\
\int_0^T \big[ \bold{VIII}(s)+\bold{IX}(s) +\bold{X}(s) \big]ds &= \; 2\nu \left(\boldsymbol{K}^* \boldsymbol{\Psi}_t\right)(g_t)(t)   dW_t.
\end{align}
Therefore, summing the above, plugging in \eqref{eq:leibniz}, using Lemma~\ref{L:Psi}-\ref{L:Psi2} and recalling \eqref{eq:linkbar} and  that $b=\kappa +u\rho \nu  $ and $\bar g_t(t)= X_t$ yield 
\begin{align}
d\langle g_t, \boldsymbol{\Psi}_t g_t \rangle_{L^2} &=  \left(-a X_t^2 +   2\kappa X_t (\boldsymbol{K}^*\boldsymbol{\Psi}_t g_t) (t) - 2X_t(\bar{\boldsymbol{\Psi}}_t g_t)(t)  \right) dt \\
&\quad + \left(  \langle g_t, \dot{\boldsymbol{\Psi}}_t g_t\rangle_{L^2} - \mbox{Tr}\Big( \boldsymbol{\Psi}_t \dot{\boldsymbol{\Sigma}}_t \Big) \right) dt  
+ 2\nu \left(\boldsymbol{K}^* \boldsymbol{\Psi}_t\right)(g_t)(t)   dW_t  \\
&= \left(-a X_t^2   - 2 u\rho  \nu X_t \left( \boldsymbol{K}^* \boldsymbol{\Psi}_t\right)(g_t)(t)   - \mbox{Tr}\Big( \boldsymbol{\Psi}_t \dot{\boldsymbol{\Sigma}}_t \Big) +  \langle g_t, \dot{\boldsymbol{\Psi}}_t g_t\rangle_{L^2}\right) dt  \\
&\quad  
+ 2\nu \left(\boldsymbol{K}^* \boldsymbol{\Psi}_t\right)(g_t)(t)   dW_t  
\end{align}
leading to  the claimed dynamics \eqref{eq:dynamicsL2inner}. 

	\end{proof}

We can now complete the proof of Theorem~\ref{T:volterra}. We recall that $\phi$ given in \eqref{eq:phii} solves
\begin{align}\label{eq:riccati_phi}
\dot {\phi}_t  =  \Tr(\boldsymbol{\Psi}_t{\boldsymbol{\dot{{\Sigma}}}_t}).
\end{align}

	\begin{proof}[Proof of Theorem~\ref{T:volterra}]
		{It suffices to prove  that \eqref{eq:charvolterra} holds for all $0\leq u\leq 1$ and $w\leq 0$ to obtain the claimed expression by analytic continuation. Indeed,  the left hand side in \eqref{eq:charvolterra} is analytic in $(u,w)$ in an open region  $(\Re(u), \Re(w)) \in (u_-,u_+)\times(w_-,w_+)$  by general results on the analycity of characteristic functions, see \citet[Theorem II.5a]{widder2015laplace}. The right hand side is also analytic in $(u,w)$ 	since  resolvents are analytic: they are given by power series. Therefore, if  \eqref{eq:charvolterra} holds  for all $0\leq u\leq 1$ and $w\leq 0$, then by analytic continuation \eqref{eq:charvolterra} remains valid on $\{ (u,w)\in \mathbb C^2: 0\leq \Re(u) \leq 1 \mbox{ and } \Re(v)\leq0\}$.} \\
	  Fix $u \in [0,1]$, $w\in \R_-$. Set 
	\begin{align}\label{eq:Uproof}
	U_t = u\log S_t + w \int_0^t X_s^2 ds + \phi_{t}  + \langle g_t, \boldsymbol{\Psi}_t g_t \rangle_{L^2},
	\end{align}
		and   $M_t = \exp(U_t)$. It suffices to prove that $M$ is a martingale. Indeed, if this is the case, then observing that the terminal value of $M$ is 
		$$M_T=u\log S_T + w\int_0^T X_s^2 ds$$ and writing the martingale property $\E[M_T|\mathcal F_t]=M_t$, for $t\leq T$,  yields \eqref{eq:charvolterra}.\\ 
		\noindent \textit{Step 1. We prove that $M$  is a local martingale by expliciting its dynamics.} We first observe that
			\begin{align}
		\label{eq:Gamma_ito_exp}
		dM_t =& M_t \big(d U_t + \frac 1 2 d\langle U \rangle_t \big).
		\end{align}
		Using \eqref{eq:steinsteinS}, we have 
		\begin{align}
		d\log S_t = -\frac 1 2 X_t^2 dt + \rho X_t dW_t + \sqrt{1-\rho^2} X_t dW_t^{\perp}.
		\end{align}
Combined with the dynamics \eqref{eq:dynamicsL2inner} and the fact that $a=w + \frac 1 2 (u^2-u)$, we get that 
\begin{align}
dU_t &=  \Big(\langle g_t, \dot{\boldsymbol{\Psi}}_t g_t \rangle_{L^2} -\frac {u^2} 2 X_t^2 - 2 u\rho  \nu X_t \left( \boldsymbol{K}^* \boldsymbol{\Psi}_t\right)(g_t)(t)  + \dot \phi_t  - \Tr\big( \boldsymbol{\Psi}_t \dot{\boldsymbol{\Sigma}}_t\big) \Big) dt \\
&\quad \quad   +  \left(\rho u X_t + 2\nu \left(\boldsymbol{K}^* \boldsymbol{\Psi}_t\right)(g_t)(t)  \right) dW_t + u \sqrt{1-\rho^2} X_t dW_t^{\perp},
\end{align}
so that 
\begin{align}
d\langle U \rangle_t =  \left(u^2 X_t^2  + 4 \rho u \nu X_t  \left(\boldsymbol{K}^* \boldsymbol{\Psi}_t\right)(g_t)(t)  + 4\nu^2  \left( \left(\boldsymbol{K}^* \boldsymbol{\Psi}_t\right)(g_t)(t)\right)^2  \right)  dt.
\end{align}
Observing that 
$$ 4\nu^2  \left( \left(\boldsymbol{K}^* \boldsymbol{\Psi}_t\right)(g_t)(t)\right)^2= -4\langle g_t , \boldsymbol{\Psi}_t \boldsymbol{\dot{\Sigma}}_t\boldsymbol{\Psi}_t g_t  \rangle_{L^2}, $$
we get that the drift part in \eqref{eq:Gamma_ito_exp} is given by 
$$ M_t \left(  \langle g_t, \left(\dot{\boldsymbol{\Psi}}_t -2  \boldsymbol{\Psi}_t \boldsymbol{\dot{\Sigma}}_t\boldsymbol{\Psi}_t\right) g_t \rangle_{L^2}   + \dot \phi_t  - \Tr\big( \boldsymbol{\Psi}_t \dot{\boldsymbol{\Sigma}}_t\big)\right) =0, $$
by virtue of the Riccati equations \eqref{eq:riccati_psiBold} and \eqref{eq:riccati_phi}.
This shows that $M$ is a local martingale. \\
 \textit{Step 2. It remains to argue  that the local martingale $M$ is a true martingale.}  To this end, we fix $t\leq T$. An application of the second part of Lemma~\ref{eq:Psiwelldefined} yields that $\boldsymbol{\Psi}_t$ is a symmetric nonpositive operator so that, recall \eqref{eq:riccati_phi},   
 $$ \langle g_t, \boldsymbol{\Psi}_t g_t \rangle_{L^2}\leq 0 \quad  \mbox{ and } \quad \phi_t = -\int_t^T \Tr(\boldsymbol{\Psi}_s \dot{\boldsymbol{\Sigma}}_s)ds \leq 0.$$
 Whence, since $w\leq 0$ and $0\leq u \leq 1$, it follows from \eqref{eq:Uproof} that 
 \begin{align}
  U_t &\leq  u\log S_t\\
  &= u \log S_0 - \frac u 2 \int_0^t X^2_s ds +  u \int_0^t X_s dB_s\\
  &\leq  u\log S_0 - \frac {u^2} 2 \int_0^t X^2_s ds +  u \int_0^t X_s dB_s
 \end{align}
Therefore, 
\begin{align*}
|M_t|=\exp(U_t) \leq \exp(u\log S_t) \leq N_t
\end{align*}
with $N_t =   S_0^u \exp\left(- \frac {u^2} 2 \int_0^t X^2_s ds +  u \int_0^t X_s dB_s\right)$ which can be shown to be a  true martingale by a similar argument to that used in \citet[Lemma 7.3]{AJLP17}. Finally, we have showed that the local martingale $M$ is bounded  by a martingale, which gives that $M$ is also a true martingale. The proof is complete. 
\end{proof}

\small
\bibliographystyle{apa}
\addcontentsline{toc}{section}{References}
\bibliography{bibl}


\end{document}